\documentclass[10pt]{amsart}
\usepackage{verbatim, latexsym, amssymb, amsmath,color}
\usepackage{hyperref}
\usepackage{epsfig}
\usepackage{bm}
\usepackage{enumerate}

\numberwithin{equation}{section}
\newtheorem{theorem}{Theorem}[section]

\newtheorem{proposition}[theorem]{Proposition}
\newtheorem{lemma}[theorem]{Lemma}
\newtheorem{corollary}[theorem]{Corollary}

\newtheorem{claim}[]{Claim}

\theoremstyle{definition}
\newtheorem{definition}[theorem]{Definition}

\newtheorem{remark}[theorem]{Remark}

\newcommand{\Ac}{\mathcal{A}^c}
\newcommand{\al}{\alpha}

\newcommand{\V}{\mathcal{V}}
\newcommand{\R}{\mathbb{R}}
\newcommand{\N}{\mathbb{N}}
\newcommand{\mH}{\mathcal{H}}
\newcommand{\F}{\mathcal{F}}

\newcommand{\sB}{\widetilde{B}}
\newcommand{\C}{\mathcal{C}}

\newcommand{\lc}{\scalebox{1.8}{$\llcorner$}}
\newcommand{\La}{\Lambda}

\newcommand{\Si}{\Sigma}
\newcommand{\de}{\delta}

\newcommand{\ep}{\epsilon}

\newcommand{\Om}{\Omega}

\newcommand{\wti}{\widetilde}

\newcommand{\mZ}{\mathbb{Z}}

\newcommand{\mS}{\mathcal{S}}

\newcommand{\M}{\mathbf{M}}
\newcommand{\bL}{\mathbf{L}}
\newcommand{\bI}{\mathbf{I}}
\newcommand{\mf}{\mathbf{f}}

\newcommand{\mF}{\mathbf{F}}
\newcommand{\mR}{\mathcal{R}}

\newcommand{\X}{\mathfrak{X}}

\newcommand{\btau}{\boldsymbol\tau}
\newcommand{\bleta}{\boldsymbol{\eta}}
\newcommand{\n}{\mathbf{n}}

\newcommand{\md}{\mathbf{d}}

\newcommand{\rom}[1]{\expandafter\romannumeral #1}
\newcommand{\Rom}[1]{\uppercase\expandafter{\romannumeral #1}}

\newcommand{\VarTan}{\operatorname{VarTan}}
\newcommand{\spt}{\operatorname{spt}}

\newcommand{\Div}{\operatorname{div}}

\newcommand{\vol}{\operatorname{Vol}} %V

\newcommand{\Area}{\operatorname{Area}}

\newcommand{\Clos}{\operatorname{Clos}}
\newcommand{\interior}{\operatorname{int}}

\title[The Existence of G-Invariant constant mean curvature Hypersurfaces]{The Existence of G-Invariant constant mean curvature Hypersurfaces}
\author[Zhiang Wu]{Zhiang Wu}
\address{Department of Mathematics, Peking University, Haidian, Beijing}
\email{wuzhiang@pku.edu.cn}

\author[Tongrui Wang]{Tongrui Wang}
\address{Department of Mathematics, Nanjing University, Gulou, Nanjing}
\email{DZ1721003@smail.nju.edu.cn}

\begin{document}
\maketitle
\begin{abstract}
  In this paper, we consider a closed Riemannian manifold $M^{n+1}$ with dimension $3\leq n+1\leq 7$, and a compact Lie group $G$ acting as isometries on $M$ with cohomogeneity at least $3$. Suppose  the union of non-principal orbits $M\setminus M^{reg}$ is a smooth embedded submanifold of $M$ without boundary and ${\rm dim}(M\setminus M^{reg})\leq n-2 $.
  Then for any $c\in\mathbb{R}$, we show the existence of a nontrivial, smooth, closed, $G$-equivariant almost embedded $G$-invariant hypersurface $\Sigma^n$ of constant mean curvature $c$.
\end{abstract}

\vspace{0.5cm}
\section{Introduction}
Given an $(n+1)$-dimensional closed, oriented, smooth Riemannian manifold $(M^{n+1},g_{_M})$, constant mean curvature (CMC) hypersurfaces are critical points of the area functional amongst variations preserving the enclosed volume. CMC hypersurfaces is a classical and important topic in differential geometry. The existence problem for CMC hypersurfaces has been studied from a number of perspectives.

In the case $n=2$, Heinz \cite{Heinz54,Heinz69}, Hildebrandt \cite{Hildebrandt70}, Struwe \cite{struwe85, struwe86, struwe88},Brezis-Coron \cite{Brezis and Coron}, etc, studied the boundary value problems of CMC hypersurfaces using the mapping method. For the case of $n\geq3$, Duzaar-Steffen \cite{Duzaar-Steffen96}, Duzaar-Grotowski \cite{Duzaar-Grotowski00} studied the boundary value problems of CMC hypersurfaces using the geometric measure theory, while both methods can only produce CMC hypersurfaces whose mean curvatures satisfy a certain upper bound.

For the case of closed CMC hypersurfaces, the classical approach is to solve the isoperimetric problem for a given volume. It is well-known that there exists a smooth minimizer for each fixed volume except for a set of Hausdorff dimension at most $n-7$, see for instance \cite{Mo03}.  However, this approach does not yield any control on the value of the mean curvature. Xin Zhou and Jonathan J. Zhu developed a min-max theory for the construction of constant mean curvature hypersurfaces of prescribed mean curvature $c$ in \cite{zhou and zhu} which considered the Lagrange-multiplier functional
\begin{equation}
\label{E:Ac0}
\mathcal{A}^c = \Area - c\vol.
\end{equation}
It is obvious that a hypersurface has constant mean curvature $c$  if it is a critical point of the functional $\Ac$. Recently, they showed that the CMC min-max theory can be extended to construct min-max prescribed mean curvature hypersurfaces for certain classes of prescription function, including a generic set of smooth functions, and all nonzero analytic functions in \cite{zhou and zhu 20}.

We want to consider when a Lie group $G$ acts by isometries on $M$, can the CMC min-max theory generate a CMC hypersurface invariant under the action of $G$?
A similar problem for minimal hypersurfaces has been studied by many people.
To the author's knowledge, J. Pitts and J. H. Rubinstein firstly announced a version of the equivariant min-max for finite group acting on $3$-dimensional manifolds in \cite{Pitts-Rubinstein} and \cite{Pitts-Rubinstein-2}. Inspired by Ketover's work, a more general version of $G$-invariant min-max under the smooth sweepouts settings in \cite{CD03} and \cite{De2013The} was built by Z. Liu in \cite{Liu} to prove the existence of $G$-invariant smooth minimal hypersurface on manifolds with dimension $3\leq n+1\leq 7$. After that, T.R Wang in \cite{WTR} built the $G$-equivariant min-max theory under the Almgren-Pitts settings to prove that there are infinitely many minimal hypersurfaces on positive Ricci curvature manifolds with dimension $3\leq n+1\leq 7$ and ${\rm dim}(M\setminus M^{reg})\leq n-2 $.

We resolve the existence of closed $G$-invariant CMC hypersurfaces, based on the min-max theory developed in \cite{zhou and zhu} for CMC hypersurfaces. In particular, we prove that:

\begin{theorem}
\label{thm:main theorem}
Let $2\leq n\leq 6$, and $(M^{n+1}, g_{_M})$ be an $(n+1)$-dimensional smooth, closed Riemannian manifold with a compact Lie group $G$ acting as isometries of cohomogeneity $Cohom(G)\geq 3$.
	Suppose  the union of non-principal orbits $M\setminus M^{reg}$ is a smooth embedded submanifold of $M$ without boundary and $dim(M\setminus M^{reg})\leq n-2 $.
Then for any $c\in\mathbb{R}$, there exists a nontrivial, smooth, closed, $G$-equivariant almost embedded $G$-invariant hypersurface $\Sigma^n$ of constant mean curvature $c$.
\end{theorem}

\medskip

Our work is inspired by the approach of X. Zhou and  J. J. Zhu and Z. Liu.
However, there are some difficulties when dealing with $G$-invariant elements.
Since we only consider the Caccioppoli sets which are invariant under the action of $G$, the variations, in this case, should keep being $G$-invariant.
Hence, due to the $G$-invariant restrictions, it seems that our variation result is weaker than the classical one which takes all kinds of variations into account.
There are already some approaches dealing with this kind of difficulty, see for example \cite{ketover} \cite{Liu} \cite{WTR}.
The idea is using an averaging argument to make the usual vector field or current to $G$-invariant one.
Combining the approach in \cite{zhou and zhu} under $G$-invariant restrictions with the averaging arguments, we show the existence of a nontrivial $G$-invariant integral varifold as a weak solution of $G$-invariant prescribed constant mean curvature $c$-hypersurfaces.

Since we are dealing with tubes around orbit rather than balls, there are also some technical problems coming up in the proof of the regularity result.
In the paper \cite{zhou and zhu}, since the $c$-min-max varifold has replacements in small annuli and the annuli in tangent space $T_pM$ are closed to annuli in manifold $M$, the regularity of tangent cones as well as various blowups of the $c$-min-max varifold can be built by the regularity result for varifolds with good replacement property \cite[Appendix C]{zhou and zhu}.
However, under the $G$-invariant restrictions, we only have replacements in $G$-sets for $(G,c)$-min-max varifold, and the $G$-annuli in manifold $M$ can be vastly different to the annuli in tangent space $T_pM$ (since $T_pM$ is only a $G_p$-space not a $G$-space).

To overcome this problem, we show the splitting property of tangent cones as well as various blowups of the $c$-min-max varifold, which has been used in \cite[Appendix 2]{Liu} and \cite[Lemma 6.5]{WTR}.
Due to the splitting property, we only need to show the regularity of tangent cones or various blowups in the normal space ${\bf N}_p(G\cdot p)$.
As for the normal part of tangent cones, we prove it is invariant under the action of the isotropy group $G_p$ before we apply the argument in \cite[Proposition 5.11]{zhou and zhu} with the good replacement property and regularity in \cite[Appendix C]{zhou and zhu} being replaced by the result in \cite[Proposition 6.2]{Liu}.
When it comes to the normal part of various blowups, we add the assumption that the union of non-principal orbits $M\setminus M^{reg}$ is a smoothly embedded submanifold of $M$ without boundary and ${\rm dim}(M\setminus M^{reg})\leq n-2 $ since we do not know whether it is $G_p$-invariant or not.
%every non-principal orbit is isolated to make sure that the various points $p_i\to p$ are all in principal orbits.
Thus we can consider $M\setminus M^{reg}$ as a whole, and the argument in the proof of \cite[Lemma 5.10]{zhou and zhu} would carry over again using the good $G$-replacements and \cite[Proposition 6.2]{Liu}.

\subsection{Outline}{$~$}

In section \ref{preliminary}, we describe some basic notations and some useful propositions of $G$-invariant currents and varifolds including the compactness theorem and the equivalence between $c$-bounded first variation and $(G,c)$-bounded first variation for varifold.

In section \ref{CMC min-max}, we define the $G$-invariant min-max theory under the Almgren-Pitts setting dealing with G-invariant objects in our paper. And we can prove the existence of nontrivial $G$-sweepouts (Theorem \ref{T:existence of nontrivial G-sweepouts}). Then for any critical sequence, we can extract a min-max sequence $\{\partial\Om^i_{x_i}\}$ that converges in the measure-theoretic sense to a nontrivial varifold $V$.

In section \ref{Tightening1}, we first recall the tightening map adapted to the $\Ac$ functional constructed in \cite{zhou and zhu}. In the remainder of section \ref{Tightening1} we will prove that after applying the modified tightening map to a critical sequence $S=\{\varphi_i\}_{i\in\N}$, every element $V$ in the critical set $C(S)$ has $c$-bounded first variation.

In section \ref{(G,c)-Almost Minimizing}, we introduce the notion of $(G,c)$-almost minimizing varifolds, and show the existence of $V\in C(S)$ which is $(G,c)$-almost minimizing in small {\em regular} annuli with $c$-bounded first variation.

In section \ref{Regularity for $(G,c)$-min-max varifold}, we  first show the existence of $(G,c)$-replacement for $(G,c)$-almost minimizing varifolds, then we show the regularity for $(G,c)$-replacement. Finally our regularity result for $(G,c)$-min-max varifold (Theorem \ref{T:main-regularity}) follows similarly to \cite[Section 6]{zhou and zhu}.

\medskip

\section{Preliminary}\label{preliminary}

In this paper, we always assume $(M^{n+1},g_{_M})$ to be a closed, oriented, smooth Riemannian manifold of dimension $3\leq (n+1)\leq 7$.
Assume that $G$ is a compact Lie group acting as isometries (of cohomogeneity $l={\rm Cohom}(G)\geq 3$) on $M$.
Let $\mu$ be a bi-invariant Haar measure on $G$ which has been normalized to $\mu(G)=1$.
By the main theorem of \cite{moore}, there is an orthogonal representation of $G$ on some Euclidean space $\R^L$, $L\in\N$, and an isometric embedding from $M$ into $\R^L$ which is $G$-equivariant.

\subsection{Basic notation}\label{notation}{~}

In this section, we collect some definitions and notions from geometric measure theory. We follow the notations from \cite{pitts}, \cite{zhou and zhu} and \cite{Liu}.
Since we are constantly dealing with G-invariant objects in our paper, we will sometimes add $G$- in front of objects meaning they are $G$-invariant:
  \begin{itemize}
    \item a $G$-varifold $V$ satisfies $g_{\#} V=V$ for all $g\in G.$
    \item a $G$-current $T$ satisfies $g_{\#} T=T$ for all $g\in G.$
    \item a $G$-vector field $X$ satisfies $g_{*} X=X$ for all $g\in G.$
    \item a $G$-map $F$ satisfies $g^{-1}\circ F\circ g=F,\forall g\in G$, (i.e. $F$ is $G$-equivariant).
    \item a $G$-set ($G$-neighborhood) is an (open) set which is a union of orbits.
  \end{itemize}
We denote by $\mH^k$ the $k$-dimensional Hausdorff measure; $\X(M)$ the space of  smooth  vector fields in $M$. $B_r(p)$, $\Clos(B_r(p))$  denote respectively the open or closed Euclidean ball of $\R^L$ and $\sB_r(p)$, $\Clos(\sB_r(p))$ denote respectively the open or closed geodesic ball of $(M,g_{_M})$.
We will also sometimes add a subscript or superscript $G$ to signify $G$-invariance just like in \cite{Liu}:
	\newcommand{\an}{\textnormal{An}^G}
	\newcommand{\ann}{\mathcal{AN}^G}
	\begin{itemize}
		\item $\pi$: the projection $\pi:M\mapsto M/G$ defined by $p \mapsto [p].$
		\item $B_\rho^G(p)$: open tubes with radius $\rho$ around the orbit $G\cdot p$ in $\R^L$.
		\item $\sB_\rho^G(p)$: open tubes with radius $\rho$ around the orbit $G\cdot p$ in $M$.
		\item $\mathfrak{X}^G(M)$: the space of smooth $G$-vector fields on $M$.
		\item $\an(p,s,t)$: the open tube $B_t^G(p)\setminus \Clos(B^G_s(p))$.
		\item $\ann_r(p)$: the set $\{\an(p,s,t):0<s<t<r\}.$
		\item $T_q(G\cdot p)$: the tangent space of the orbit $G\cdot p$ at some point $q\in G\cdot p$.
		\item ${\bf N}_p(G\cdot p)$: the normal vector space of the orbit $G\cdot p$ in $M$ at point $p\in G\cdot p$.
		%\item ${\bf N}(G\cdot p)$: the normal vector bundle of the orbit $G\cdot p$ in $M$.
		\item ${\rm dim}_p$: the dimension of the orbit $G\cdot p$ and ${\rm dim}_p\leq n+1-l$.
	\end{itemize}

\medskip

The spaces we will work with in this paper are:
\begin{itemize}
\setlength{\itemsep}{5pt}
\item  $\bI_{k}(M)$ ($\bI_{k}^G(M)$) the space of $k$-dimensional ($G$-invariant) integral currents in $\mathbb{R}^L$ with support contained in $M$ (see \cite[4.2.26]{federer} for more details);
\item  ${\mathcal Z}_k(M)$ (${\mathcal Z}_k^G(M)$) the space of ($G$-invariant)  integral currents  $T \in \bI_{k}(M)$ ($T \in \bI_{k}^G(M)$) with $\partial T=0$
\item $\mathcal{V}_k(M)$ ($\mathcal{V}_k^G(M)$) the closure, in the weak topology, of the space of $k$-dimensional ($G$-invariant) rectifiable varifolds in $\mathbb{R}^L$ with support contained in $M$. The space of integral ($G$-invariant) rectifiable $k$-dimensional varifolds with support contained in $M$ is denoted by $\mathcal{IV}_k(M)$ ($\mathcal{IV}_k^G(M)$).
\item  $\C(M)$ ($\C^G(M)$) the space of ($G$-invariant) sets $\Om\subset M$  with finite perimeter (Caccioppoli sets), \cite[\S 14]{simon}.
\end{itemize}

Just like in \cite{zhou and zhu} we also utilize the following definitions:
\begin{enumerate}[(a)]
\label{En: notations}
\item Given $T\in\bI_{k}(M)$ , $|T|$ and $\|T\|$ denote respectively the integral varifold and Radon measure in $M$ associated with $T$;
\item $\F$ and $\M$ respectively the flat norm \cite[\S 31]{simon} and mass norm \cite[26.4]{simon} on $\bI_k(M)$; The varifold ${\bf F}$-{\it metric} on $\mathcal{V}_k(M)$ is defined in  {Pitts's book} \cite[P.66]{pitts}, which induces the varifold weak topology on $\mathcal{V}_k(M) \cap \{V:\|V\|\leq C \}$ for all $C$. The current ${\bf F}$-{\it metric} on $\bI_{k}(M)$ is defined by  $${\bf F}(S,T)=\F (S-T)+{\bf F}(|S|,|T|);$$
    Keep in mind that the mass $\M$ is continuous in the current ${\bf F}$-{\it metric} but not in the flat topology, and  the current ${\bf F}$-{\it metric} satisfies $$\F (S-T)\leq {\bf F}(S,T)\leq 2\M(S-T);$$
\item Given $c>0$, a varifold $V\in \V_k(M)$ is said to have {\em $c$-bounded first variation in an open subset $U\subset M$}, if
\[ |\de V(X)|\leq c \int_M|X|d\mu_V, \quad \text{for any } X\in\X(U); \]
here the first variation of $V$ along $X$ is $\de V(X)=\int_{G_k(M)} \Div_S X(x)d V(x, S)$;
%\item $U_r(V)$ denotes the ball in $\V_k(M)$ under $\mF$-metric with center $V\in\V_k(M)$ and radius $r>0$;
%\item Given a varifold $V\in \V_k(M)$ and $p\in\spt\|V\|$, $\VarTan(V,p)$ denotes the space of tangent varifolds of $V$ at $p$, \cite[42.3]{Si83};
\item Given a set $\Om\in\C(M)$ ($\C^G(M)$) with finite perimeter, $[[\Om]]$  denotes the corresponding ($G$-invariant) integral currents with the natural orientation.  $|\partial\Om|$ and $\partial\Om$ denote respectively the corresponding ($G$-invariant) integral rectifiable varifold and reduced-boundary of $[[\Om]]$ as an ($G$-invariant) integral current. When the boundary $\Si=\partial\Om$ is a smooth immersed hypersurface, we have \[\Div_{\Si}X=H\,  X\cdot \vec{\nu},\] where  $\vec{\nu}=\vec{\nu}_{\partial\Om}$ is the outward pointing unit normal of $\partial \Om$, \cite[14.2]{simon} and $H$ is the mean curvature of $\Si$ with respect to $\vec{\nu}$.
%\item  Given a set $\Om\in\C(M)$ ($\C^G(M)$) with finite perimeter, $|\partial\Om|$ denotes the corresponding integral rectifiable ($G$-invariant) varifolds;
%If $\Om\in \C^G(M)$, then $\nu_{\partial\Om}$ is a $G$-invariant vector field on $\partial \Om$.
\end{enumerate}

\medskip
In \cite{zhou and zhu},  Xin Zhou and Jonathan J. Zhu developed a min-max theory for the construction of constant mean curvature hypersurfaces of prescribed mean curvature $c$ in $(M,g_{_M})$. They studied the following {\em $\Ac$-functional} defined on $\C(M)$ for $c>0$
\begin{equation}
\label{E: Ac}
\Ac(\Om)=\mH^n(\partial\Om)-c\mH^{n+1}(\Om).
\end{equation}
The {\em first variation formula} for $\Ac$ along $X\in \X(M)$ is
\begin{equation}
\label{E: 1st variation for Ac}
\de\Ac|_{\Om}(X)=\int_{\partial\Om}\Div_{\partial \Om}X d\mu_{\partial\Om}-c\int_{\partial\Om}X\cdot \vec{\nu} \, d\mu_{\partial\Om},
\end{equation}

When $\Om$ is a critical point of $\Ac$ and $\Si=\partial\Om$ is a smooth immersed hypersurface, then (\ref{E: 1st variation for Ac}) directly implies that $\Si$ has constant mean curvature $c$ with respect to the outward unit normal $\vec{\nu}$. In this case, we can calculate the {\em second variation formula} for $\Ac$ along normal vector fields $X\in \X(M)$ such that $X=\varphi\vec{\nu} $ along $\Si$ where $\varphi\in C^\infty(\Si)$ (see \cite[Proposition 2.5]{BCE88}),
\begin{equation}
\label{E: 2nd variation for Ac}
\de^2\Ac|_{\Om}(X,X) = \Rom{2}_\Si (X,X) =\int_{\Si}\left( |\nabla^\Si\varphi|^2-\left(Ric^M(\vec{\nu}, \vec{\nu})+|A^\Si|^2\right)\varphi^2\right)d\mu_{\Si}.
\end{equation}
In the above formula, $\nabla^\Si\varphi$ is the gradient of $\varphi$ on $\Si$; $Ric^M$ is the Ricci curvature of $M$; $A^\Si$ is the second fundamental form of $\Si$.

\subsection{$G$-invariant currents and varifolds}\label{Sec-G-current}{~}

 We will deal with G-invariant objects in our paper. Now we collect some useful propositions for $G$-currents and $G$-varifolds from \cite{WTR} and \cite{Liu} .

\begin{proposition}[Compactness Theorem for ${\bf I}_k^G(M)$ ]\label{G-current compactness}
	For any $C>0$, the set:
$$\{ T\in {\bf I}_k^G(M): {\bf M}(T)+{\bf M}(\partial T) \leq C\} $$ is compact under the flat metric $\mathcal{F}$.
\end{proposition}

\begin{proposition}[Compactness Theorem for $\mathcal{V}^G_k(M)$ ]\label{G-varifold compactness}
For any $C>0$, the set:
$$\{ V\in \V_k^G(M): ||V||(M) \leq C\} $$ is compact under the  ${\bf F}$-{\it metric} of varifolds.
\end{proposition}

\begin{proposition}[Restrict to $G$-sets]\label{restrict}
	For any $V\in\mathcal{V}^G_k(M)$ and Borel $G$-set $U$, we have $V|_U\in\V^G_k(M)$.
\end{proposition}

\begin{proposition}[$G$-equivariant pushing forward]\label{Prop-G-push forward}
	Let $\Phi$ be a $G$-equivariant diffeomorphism of $M$,
	then for $T\in {\bf I}_k^G(M)$ and $V\in\mathcal{V}^G_k(M)$, the pushing forward  $\Phi_\#T \in {\bf I}_{k}^G(M)$ and $\Phi_\#V\in\mathcal{V}^G_k(M)$.
\end{proposition}

\begin{proposition}[$G$-invariant slice]\label{slice}
	Let $f$ be a $G$-invariant Lipschitz function on $M$ and $T\in {\bf I}_k^G(M)$, then for almost all $t\in\mathbb{R}$, the slice of $T$ by $f$ at $t$ exists and $\langle T,f,t\rangle \in {\bf I}_{k-1}^G(M)$.
\end{proposition}

\begin{proposition}[$G$-invariant Isoperimetric Lemma]\label{isoperimetric}
	There exist positive constants $\nu_M$ such that for any $T_1,T_2\in {\mathcal Z}_n^G(M)$ with
$$\mathcal{F}(T_1-T_2)<\nu_M $$
there exists a unique $Q\in {\bf I}_{n+1}^G(M)$ such that:
\[ \partial Q = T_1-T_2, \quad \text{} {\bf M}(Q) = \mathcal{F}(T_1-T_2). \]	
\end{proposition}

\begin{remark}\label{G-v-closed}
	An important case for $G$-equivariant diffeomorphisms is the one-parameter group of diffeomorphisms generated by $X\in \mathfrak{X}^G(M)$. A simple example for $G$-invariant slice is taking $f$ to be the distant functions to orbits ${\rm dist}(G\cdot x,\cdot)$.
\end{remark}

\begin{definition}[$(G,c)$-bounded first variation]\label{Def-bounded first variation}
	Given $c>0$, a varifold $V\in \V^G_k(M)$ is said to have {\em $(G,c)$-bounded first variation in a $G$-open subset $U\subset M$}, if
\[ |\de V(X)|\leq c \int_M|X|d\mu_V, \quad \text{for any } X\in \mathfrak{X}^G(U). \]
\end{definition}

\medskip
In \cite[Lemma 2.2]{Liu}, Z. Liu has shown the equivalence between $G$-stationary and stationary for $G$-varifolds.  We notice that Liu's arguments also imply the equivalence between the $(G,c)$-bounded first variation and $c$-bounded first variation for $G$-varifolds.

\begin{lemma}\label{$c$-bounded first variation}
 For any $V\in \V^G_k(M)$, and $G$-open subset $U\subset M$, then $V$ has $(G,c)$-bounded first variation in $U$ if and only if $V$ has $c$-bounded first variation in $U$.
\end{lemma}

\begin{proof}
	If $V$ has $c$-bounded first variation in $U$, it's clear that $V$ has $(G,c)$-bounded first variation in $U$ by definitions.
	Suppose now $V$ has $(G,c)$-bounded first variation in $U$.
It follows from \cite[Lemma 2.2]{Liu} that for any vector field $X\in\X(M)$, with $|X|\leq 1$ supported in $U$, there exists a $G$-vector field $X_G$ supported in $U$, with $|X_G|\leq |X|$ so that,
\[ \de V(X)=\de V(X_G). \]

Since $V$ has $(G,c)$-bounded first variation in $U$, we have
\[|\de V(X_G)|\leq c \int_M|X_G|d\mu_V\leq c \int_M|X|d\mu_V. \]
Hence
\[|\de V(X)|=|\de V(X_G)|\leq c \int_M|X|d\mu_V. \]
This implies that
	$V$ has $c$-bounded first variation in $U$.
\end{proof}

\subsection{Compactness of $G$-stable CMC $G$-hypersurfaces}
\label{SS:Compactness of $G$-stable CMC hypersurfaces}{~}

\begin{definition}[$G$-stable $(G,c)$-hypersurface]\label{D:G-stable c-hypersurface}
Let $\Si$ be a smooth, immersed, two-sided,  hypersurface with unit normal vector $\vec{\nu}$. We say that $\Si$ is a {\em $G$-stable $(G,c)$-hypersurface} in an open $G$-subset $U\subset M$ if
\begin{itemize}
\item  $\Si$ is  $G$-invariant;
\item  {the mean curvature $H$ of $\Si\cap U$ with respect to $\vec{\nu}$ equals to $c$; }
\item  $\Rom{2}_\Si(X, X)\geq 0$ { for all $G$-invariant vector field $X\in\mathfrak{X}^{G,\bot}(\Si\cap U)$}.
\end{itemize}
\end{definition}

In the above definition, the prefix `$G$-' of `stable' means that only $G$-invariant vector fields are considered for variations.
While the prefix `$G$-' of `hypersurface' means the hypersurface is invariant under the action of $G$.
Clearly, the prefix `$c$-' of `hypersurface' implies that the hypersurface has constant mean curvature which equals $c$ with respect to the given unit normal vector field.
With these in mind, one can similarly define stable $(G,c)$-hypersurfaces as well as stable $c$-hypersurfaces.

The following result about the equivalence between the stability and $G$-stability can be found in \cite[Lemma 2.9]{WTR}. T.R Wang has shown the equivalence between the $G$-stability and stability for minimal $G$-hypersurfaces of boundary type. We notice that the second variation formula for $\Ac$ functional is the same as the second variation formula for $Area$ functional. This implies the equivalence between the stability and $G$-stability of boundary type $G$-invariant $c$-hypersurfaces.
\begin{lemma}\label{G-stable}
	Let $U\subset M$ be an open $G$-subset, and $\Sigma=\partial \Om$ the boundary of $\Om \in\C^G(U)$.
	If $\Si$ is a {\em $c$-hypersurface},
	then $\Sigma$ is $G$-stable if and only if it is stable.
\end{lemma}

\medskip
Noting that the Lie group $G$ is not assumed to be connected, a connected component of a $G$-hypersurface may not be $G$-invariant.
Hence, we introduce the following $G$-invariant analogs of the usual notions of connectivity.
\begin{definition}[$G$-connectivity]\label{D:connected}
	Suppose $\Sigma$ is a $G$-hypersurface, and $\{\Sigma_i\}_{i=1}^k$ are connected components of $\Sigma$.
	We say $\Sigma$ is {\em $G$-connected}, if for any $i,j\in\{1,\dots,k\}$, there exists $g\in G$ such that $\Sigma_i=g\cdot\Sigma_j$.
	
	Similarly, we say a $G$-set $U$ is {\em $G$-connected} if for any two connected components of $U$ there exists $g\in G$ as a diffeomorphism between them.
\end{definition}

\medskip
Since we are dealing with $G$-invariant hypersurfaces, we need to make few analogs of the definitions in \cite[Section 2.1]{zhou and zhu}.

\begin{definition}
	Let $\Sigma_i$, $i=1,2$, be connected embedded two-sided hypersurfaces in a connected open set $U\subset M$, with $\partial \Sigma_i\cap U=\emptyset$ and unit normals $\vec{\nu_i}$.
	We say that {\em $\Sigma_2$ lies on one side of $\Sigma_1$ in $U$} if $\Sigma_1$ divides $U$ into two connected components $U_1\cup U_2=U\setminus \Sigma_1$, where $\vec{\nu_1}$ points into $U_1$, and either:
	\begin{itemize}
		\item $\Sigma_2\subset \Clos(U_1)$, which we write as $\Sigma_1\leq\Sigma_2$ or that $\Sigma_2$ lies on the positive side of $\Sigma_1$; or
		\item $\Sigma_2\subset \Clos(U_2)$, which we write as $\Sigma_1\geq\Sigma_2$ or that $\Sigma_2$ lies on the negative side of $\Sigma_1$.
	\end{itemize}
\end{definition}

A manifold $N$ is said to be a smooth $G$-manifold if $N$ is a smooth manifold with $G$ acts on it smoothly.

\begin{definition}[$G$-equivariant almost embedding]\label{D:almost embedding}
	Let $U\subset M^{n+1}$ be an open $G$-set, and $\Sigma'$ be a $n$-dimensional smooth $G$-manifold.
	A smooth $G$-equivariant immersion $\phi:\Sigma'\to U$ is said to be a {\em $G$-equivariant almost embedding} if at any point $p\in \phi(\Sigma')$ where $\phi(\Sigma')$ fails to be embedded, there exists a small $G$-connected $G$-neighborhood $W\subset U$ of $p$, such that
	\begin{itemize}
		\item $\Sigma'\cap \phi^{-1}(W)$ is a disjoint union of connected components $\cup_{i=1}^l \Sigma_i'$;
		\item $\phi(\Sigma_i')$ is an embedding for each $i=1,\dots,l$;
		\item for each $i$, suppose $\phi(\Si_i')\subset W_i$ where $W_i$ is a connected component of $W$, then any other component $\phi(\Sigma_j')$, $j\neq i$, either lies on one side of $\phi(\Sigma_i')$ in $W_i$, or contains in a different connected component of $W$.
	\end{itemize}
\end{definition}
	We will denote $\phi(\Si')$ and $\phi(\Si_i')$ by $\Sigma$ and $\Si_i$ for simplicity.
	The subset of points in $\Sigma$ where $\Sigma$ fails to be embedded will be called the {\em touching set}, and denoted by $\mS(\Si)$.
	We also call $\Si\setminus\mS(\Si)$ the regular set, and denote it by $\mR(\Si)$.
	We also denote $\mS(\Si')=\phi^{-1}(\mS(\Si)),~\mR(\Si')=\phi^{-1}(\mR(\Si)) $.

Since $\phi$ is $G$-equivariant, it's clear that $\Sigma'\cap\phi^{-1}(W)=\cup_{i=1}^l\Sigma_i'$ is $G$-invariant.
But generally, we do not know whether each $\Sigma_i'$ is $G$-invariant or not.
For example, consider the $G=\mZ_2$ action on $\R^2$ as $e\cdot (x,y)=(x,y), g\cdot (x,y)=(-x,y)$.
Define $\Sigma_i=\Si_i'=(S^1+((-1)^i,0))$ to be a unit $1$-sphere with center point $(-1,0)$ or $(1,0)$.
Then $\Sigma=(S^1+(1,0))\cup (S^1+(-1,0))$ is an equivariant almost embedded $G$-hypersurface with one touching point $\mS(\Sigma)=\{(0,0)\}$.
But each $\Sigma_i'$ is not $G$-invariant since $g\cdot \Sigma_i'=\Sigma_j',~i\neq j \in\{1,2\}$.

Moreover, since $W$ is $G$-connected, we can decompose it by connected components as $\cup_{\alpha=1}^\iota W_\alpha$ such that there exists $g_{\alpha,\beta}\in G$ with $g_{\alpha,\beta}\cdot W_\beta=W_\alpha$.
Thus $\{\Sigma_i'\}_{i=1}^l$ can be classified by which component $W_\alpha$ contains its image.
We write $\{ \Si_{\alpha,i} ' \}_{i=1}^{l_\alpha}$ to be all the $\Sigma_j'$ such that $\Si_j=\phi(\Si_j')$ is contained in $W_\alpha$.
Hence $\{ \Si_{\alpha,i} ' \}_{i=1}^{l_\alpha}=g_{\alpha,\beta}\cdot \{ \Si_{\beta,i} ' \}_{i=1}^{l_\beta}$ and $l_{\alpha}=l_{\beta}$.
After choosing the component $W_\alpha$ containing the touching point $p$, it's clear by definition that the term `$G$-equivariant almost embedded' implies `almost embedded and $G$-invariant' (see \cite[Definition 2.3]{zhou and zhu} for the definition of almost embedding).

On the other hand, we can also use the slices of orbits to make the definition of $G$-equivariant almost embedding more localized.
To be exact, let $\sB_p$ to be a slice of $G\cdot p$ at $p$ in $M$ (see \cite[Section 3.3]{wall}), and $G_p=\{g\in G : g\cdot p=p\}$ be the isotropy group of $p$.

\begin{definition}[equivariant almost embedding in slices]
 A smooth $G$-equivariant immersion $\phi:\Sigma'\to U\subset M$ is said to be an {\it equivariant almost embedding in slices} if for any $p\in U$, $\phi:\Sigma'\cap \phi^{-1}(\sB_p)\to \sB_p$ is a $G_p$-equivariant almost embedding.
\end{definition}
 		
For any $G$-invariant hypersurface $\Sigma$, using the transversality of slices and orbits, we can show that $\Sigma$ is $G$-equivariant almost embedded if and only if it is equivariant almost embedded in slices.

Furthermore, we claim that the term `$G$-equivariant almost embedded' is equivalent to the term `almost embedded and $G$-invariant'.
Indeed, suppose $\phi:\Sigma'\to M$ is an almost embedding in the sense of \cite[Definition 2.3]{zhou and zhu} with $\Sigma=\phi(\Si')$ being $G$-invariant.
We are going to show that there is a $G$-structure on $\Sigma'$ so that $\phi$ is $G$-equivariant, and $\phi$ is an equivariant almost embedding in slices, which implies $\phi$ is a $G$-equivariant almost embedding by arguments in the previous paragraph.
Since $G$ acts on $M$ by isometries and $\Sigma$ is $G$-invariant, the regular set $\mR(\Si)$ as well as the touching set $\mS(\Si)$ is $G$-invariant, i.e. $p\in \mS(\Sigma)$(or $\mR(\Sigma)$) $ \Leftrightarrow$ $G\cdot p \subset \mS(\Sigma)$(or $\mR(\Sigma)$).
Thus we can define the action of $G$ on $\mR(\Si')=\phi^{-1}(\mR(\Si))$ by
$$g\cdot x=g\cdot \phi^{-1}(q)=\phi^{-1}(g\cdot q), \quad \forall q\in \mR(\Si),~ x=\phi^{-1}(q)\in\mR(\Si'),$$
since $\phi\vert_{\mR(\Si')}$ is embedding.
As for the action of $G$ on $\mS(\Si')=\phi^{-1}(\mS(\Si))$, we first take $p\in\mS(\Si)$ and use the definition \cite[Definition 2.3]{zhou and zhu} to get a connected open neighborhood $\wti{W}$ of $p$ which satisfies
\begin{itemize}
	\item $\Sigma'\cap \phi^{-1}(\wti{W})$ is a disjoint union of connected components $\cup_{i=1}^l \Sigma_i'$;
	\item $\phi(\Sigma_i')$ is an embedding for each $i=1,\dots,l$;
	\item for each $i$, any other component $\phi(\Sigma_j')$, $j\neq i$, lies on one side of $\phi(\Sigma_i')$ in $\wti{W}$.
\end{itemize}
Then we take the slice $\sB_p$ of $G\cdot p$ at $p$ in $\wti{W}$ and consider the restriction $\phi\vert_{\sB_p'}:\Sigma'\cap \phi^{-1}(\sB_p)\to \sB_p$, where $\sB_p'=\phi^{-1}(\sB_p)$.
By the transversality of slices and orbits, we can show that:
\begin{itemize}
	\item[(i)] $\Sigma'\cap \phi^{-1}(\sB_p)$ is a disjoint union of connected components $\cup_{i=1}^l \Sigma_{p,i}'$ where $\Sigma_{p,i}'= (\phi\vert_{\Sigma_i'})^{-1}(\phi(\Si_i')\cap\sB_p)$;
	\item[(ii)] $\phi(\Sigma_{p,i}')$ is an embedding in the slice $\sB_p$ for each $i=1,\dots,l$;
	\item[(iii)] for each $i$, any other component $\phi(\Sigma_{p,j}')$, $j\neq i$, lies on one side of $\phi(\Sigma_{p,i}')$ in $\sB_p$.
\end{itemize}
For any $g\in G$, denote $\sB_{g\cdot p}=g\cdot \sB_p$, $\Si_{p,i}=\phi(\Si_{p,i}')$, $\Si_{g\cdot p, i}=g\cdot \Si_{p,i}$.
Thus we have $\sB_{g\cdot p}$ is a slice of $G\cdot p $ at $g\cdot p$, and $\Si_{g\cdot p, i}\subset \sB_{g\cdot p}$ is connected embedded (since $g$ is an isometry).
Furthermore, since $\Sigma$ is $G$-invariant, it's clear that $\{ \Si_{g\cdot p,i}'=\phi^{-1}(\Si_{g\cdot p,i})\}_{i=1}^l$ satisfies the properties (i)(ii)(iii) with $\sB_{g\cdot p}$ in place of $\sB_p$.
Then for any $x\in \phi^{-1}(p)\subset\mS(\Si')$, there exists $i_x\in\{1,\dots,l\}$ so that $x\in \Si_{p,i_x}'$.
We can now define $g\cdot x=(\phi\vert_{\Si_{g\cdot p,i_x}'})^{-1}(g\cdot p) $.
This gives a $G$-structure on $\Sigma'$ which makes $\phi$ to be $G$-equivariant.
The above argument also implies that $\phi$ is an equivariant almost embedding in slices, and thus $G$-equivariant almost embedding.

\medskip
Finally, we want to point out that if the Lie group $G$ is connected then each $\Si_i=\phi(\Si_i')$ must be $G$-invariant.
Indeed, for any $g\in G$, there exists a curve $g(t):[0,1]\to G$ with $g(0)=e,~g(1)=g$ by the connectivity of $G$.
Since $\Sigma'\cap \phi^{-1}(W)=\cup_{i=1}^l\Si_i'$ is $G$-invariant, we have $g\cdot \Si_j'\subset \cup_{i=1}^l\Si_i',~\forall j\in\{1,\dots,l\}$.
However, for any $x\in\Si_j'$, there is a curve $\gamma(t)=g(t)\cdot x$ from $x$ to $g\cdot x$ which implies $g\cdot \Si_j'\subset \Si_j'$ since $\Si_j'$ is a connected component.
It means that $g\cdot \Si_j'=\Si_j'$ for any $g\in G$, $j\in\{1,\dots,l\}$. We have that $\Si_i=\phi(\Si_i')$ is $G$-invariant, since $\phi:\Sigma'\to U$ is $G$-equivariant and each $\Si_i'$ is $G$-invariant.

We now summarize these properties of $G$-almost embedding in the following remark.

\begin{remark}\label{R:almost embedded}
	Suppose $\Sigma$ is an immersed hypersurface in an open $G$-set $U\subset M$.
	\begin{itemize}
		\item[(i)] $\Sigma$ is $G$-equivariant almost embedded in $U$ if and only if $\Sigma$ is an almost embedded $G$-hypersurface in $U$.
		\item[(ii)] If $\Sigma$ is $G$-equivariant almost embedded, then the collection of components $\Sigma_i$ meet tangentially along $\mS(\Si)$ by \cite[Remark 2.4]{zhou and zhu}.
		\item[(iii)] If $\Sigma$ is $G$-equivariant almost embedded, then the touching set $\mS(\Sigma)$ as well as the regular set $\mR(\Sigma)$ of $\Sigma$ is a $G$-set, which implies that if $p\in \mS(\Sigma)$(or $\mR(\Sigma)$) then $G\cdot p \subset \mS(\Sigma)$(or $\mR(\Sigma)$) by the definition.
		\item[(iv)] $\Sigma$ is $G$-equivariant almost embedded if and only if it is equivariant almost embedded in slices.
		\item[(v)] If the Lie group $G$ is connected, then each $\Si_i'$ as well as $\Si_i$ is $G$-invariant.
	\end{itemize}	
\end{remark}

\begin{definition}[$G$-equivariant almost embedded $(G,c)$-boundary]
	Suppose $U\subset M$ is an open $G$-set.
	\begin{itemize}
		\item[(1)]  A $G$-equivariant almost embedded $G$-hypersurface $\Sigma\subset U$ is said to be a {\em $G$-equivariant almost embedded $G$-boundary} if there is an open $G$-set $\Om\in\C^G(U)$ such that $\Sigma$ is equal to the boundary of $\partial \Omega$ in the sense of currents.
		\item[(2)] The outer unit normal $\vec{\nu}_\Sigma$ of $\Si$ is the choice of the unit normal of $\Sigma$ which points outside of $\Omega$ along the regular part $\mR(\Si)$.
		\item[(3)] $\Sigma$ is called a {\em stable $(G,c)$-boundary} if $\Sigma$ is a $G$-boundary as well as a stable immersed $(G,c)$-hypersurface.
	\end{itemize}
\end{definition}

\begin{remark}\label{R:G-stable}
	By Lemma \ref{G-stable}, $\Sigma$ is a stable $(G,c)$-boundary if and only if $\Sigma$ is $G$-stable $(G,c)$-boundary. 	
\end{remark}

We will need the following compactness theorem for $G$-stable $(G,c)$-hypersurfaces which is essentially due to Xin Zhou and Jonathan J. Zhu \cite[Theorem 2.11]{zhou and zhu}.
\begin{theorem}[Compactness theorem for $G$-stable $(G,c)$-boundary]
\label{T:compactness}
	Let $2\leq n\leq 6$. Suppose $\Si_k\subset U$ is a sequence of smooth, $G$-equivariant almost embedded, $G$-stable, $(G,c_k)$-boundaries in an open $G$-set $U$, with $\sup_{k} \Area(\Sigma_k) < \infty$ and $\sup_k c_k <\infty$. Then the following hold:
\begin{itemize}
  \item[(i)]
	If $\inf_k c_k>0$, then up to a subsequence, $\{\Sigma_k\}$ converges locally smoothly to some $G$-equivariant almost embedded $G$-stable $(G,c)$-boundary $\Sigma_\infty$ in $U$, and the density of $\Sigma_\infty$ is $1$ along $\mR(\Si_\infty)$ and $2$ along $\mS(\Si_\infty)$.
  \item[(ii)] If $\inf_k c_k\to 0$, then up to a subsequence, $\{\Sigma_k\}$ converges locally smoothly (with multiplicity) to some smooth embedded stable $G$-invariant minimal hypersurface $\Sigma_\infty$ in U.
	\end{itemize}
\end{theorem}
\begin{proof}
	As we pointed out in Remark \ref{R:almost embedded}(i) and Remark \ref{R:G-stable}, each $\Sigma_k$ is an almost embedded stable $G$-invariant $c_k$-boundary.
	Hence, we can apply the compactness theorem \cite[Theorem 2.11(i)(ii)]{zhou and zhu} to get a subsequence converging locally smoothly to an almost embedded stable $c$-boundary $\Sigma_\infty$ satisfying the density requirements.
	By Remark \ref{R:almost embedded}(i), we only need to show that $\Sigma_\infty$ is $G$-boundary.
	
	Denote $\Sigma_k=\partial\Omega_k$ for some $\Omega_k\in\C^G(M)$.
	Combining Proposition \ref{G-current compactness}, \ref{G-varifold compactness} with the proof of \cite[Theorem 2.11(ii)]{zhou and zhu}, we have $\partial\Omega_k $ converges weakly as $G$-currents to some $\partial\Omega_\infty$ with $\Omega_\infty\in\C^G(M)$ and $\Sigma_\infty=\partial \Omega_\infty$ as $G$-varifold.
	Thus $\Sigma_\infty$ is $G$-invariant and $G$-equivariant almost embedded by Remark \ref{R:almost embedded}(i).
\end{proof}

\section{$G$-invariant min$-$max theory for CMC hypersurfaces}\label{CMC min-max}
\medskip

We will introduce the min-max construction using the scheme developed by Almgren and Pitts \cite{almgren}, \cite{almgren-varifolds}, \cite{pitts}. This section follows from \cite[Section 3]{zhou and zhu} with $\C^G(M)$ in place of $\C(M)$.

\subsection{Cell Complex}\label{Sec-sub-cubical-complex}{~}

For any $j\in \mathbb{N}$, we denote $I(1,j)$ to be the cell complex on $I=[0,1]$ with $1$-cells
$$[0,3^{-j}], [3^{-j},2 \cdot 3^{-j}],\dots,[1-3^{-j}, 1]$$
   and $0$-cells (vertices)
$$[0], [3^{-j}],\dots,[1-3^{-j}], [1].$$
Given a $1$-cell $\al\in I(1, j)_1$, and $k\in\N$, we will use the following notations:
\begin{itemize}
	\item $I(1, j)_p$: the set of all $p$-cells in $I(1,j)$;
	\item $I_0(1, j)_0$: the set $\{[0], [1]\}$;
	\item $\al(k)$: the sub-complex of $I(1, j+k)$ formed by all cells contained in $\al$;
\end{itemize}
We also utilize the following definitions:
\begin{itemize}
    \item  The boundary homeomorphism $\partial: I(1, j)\rightarrow I(1, j)$ is given by $\partial[a, b]=[b]-[a]$ if $[a, b]\in I(1, j)_1$, and $\partial[a]=0$ if $[a]\in I(1, j)_0$.
    \item  The distance function $\md: I(1, j)_0\times I(1, j)_0\rightarrow\N$ is defined as $\md(x, y)=3^{j}|x-y|$. We say $x,y$ are {\em adjacent} if ${\bf d}(x,y)=1$.
    \item  The map $\n(i, j): I(1, i)_{0}\to I(1, j)_{0}$ is defined as: $\n(i, j)(x)\in I(1, j)_{0}$ is the unique element of $I(1, j)_0$, such that
    $$\md\big(x, \n(i, j)(x)\big)=\inf\big\{\md(x, y): y\in I(1, j)_{0}\big\}.$$
    \item  For any map $\phi: I(1, j)_{0}\rightarrow\C^G(M)$, the \emph{fineness} of $\phi$ with respect to the $\M$ norm is defined as:
\begin{equation*}\label{fineness}
\mf(\phi)=\sup\Big\{\frac{\M\big(\partial\phi(x)-\partial\phi(y)\big)}{\md(x, y)}:\ x, y\in I(1, j)_{0}, x\neq y\Big\}.
\end{equation*}
   It was observed by Fernando C. Marques and Andr\'e Neves that $\mf(\phi)<\de$ if and only if $\M\big(\partial\phi(x)-\partial\phi(y)\big)<\de$ whenever ${\bf d}(x,y)=1$.
   Similarly, we can define the fineness of $\phi$ with respect to the $\F$-norm and $\mF$-metric.
   \item  We denote $\phi: I(1, j)_{0}\rightarrow\big(\C^G(M), \{0\}\big)$ as a map such that $\phi\big(I(1, j)_{0}\big)\subset\C^G(M)$ and $\partial\phi|_{I_{0}(1, j)_{0}}=0$, i.e. $\phi([0]), \phi([1])=\emptyset$ or $M$.
\end{itemize}

\subsection{Homotopy sequences}\label{Sec-Homotopy-sequences}{~}

\begin{definition}\label{homotopy for maps}
Given $\de>0$ and $\phi_{i}: I(1, k_{i})_{0}\rightarrow\big(\C^G(M), \{0\}\big)$, $i=1,2$, we say \emph{$\phi_{1}$ is $1$-homotopic to $\phi_{2}$ in $\big(\C^G(M), \{0\}\big)$ with fineness $\de$}, if we can find a map
$$\psi: I(1, k_{3})_{0}\times I(1, k_{3})_{0}\rightarrow \C^G(M).$$
for some $k_{3}\in\N$ and $k_{3}\geq\max\{k_{1}, k_{2}\}$ such that
\begin{itemize}
\setlength{\itemindent}{1em}
\item $\mf(\psi)\leq \de$;
\item $\psi([i-1], x)=\phi_{i}\big(\n(k_{3}, k_{i})(x)\big)$, $i=1,2$;
\item $\partial\psi\big(I(1, k_{3})_{0}\times I_{0}(1, k_{3})_{0}\big)=0$.
\end{itemize}
\end{definition}
\begin{remark}
Note that the first and the second conditions imply that $\mf(\phi_{i})\leq \de$, $i=1,2$.
\end{remark}

\begin{definition}\label{(1, M) homotopy sequence}
A \emph{$(1, \M)$-homotopy sequence of mappings into $\big(\C^G(M), \{0\}\big)$} is a sequence of mappings $\{\phi_{i}\}_{i\in\N}$,
$$\phi_{i}: I(1, k_{i})_{0}\rightarrow\big(\C^G(M), \{0\}\big),$$
such that $\phi_{i}$ is $1$-homotopic to $\phi_{i+1}$ in $\big(\C^G(M), \{0\}\big)$ with fineness $\de_{i}>0$, and
\begin{itemize}
\setlength{\itemindent}{1em}
\item $\lim_{i\rightarrow\infty}\de_{i}=0$;
\item $\sup_{i}\big\{\M(\partial\phi_{i}(x)):\ x\in I(1, k_{i})_{0}\big\}<+\infty$.
\end{itemize}
\end{definition}

\begin{remark}
Note that the second condition implies that $\sup_{i}\big\{\Ac(\phi_{i}(x)):\ x\in I(1, k_{i})_{0}\big\}<+\infty$.
\end{remark}

\begin{definition}\label{homotopy for sequences}
Given two $(1, \M)$-homotopy sequences of mappings $S_{1}=\{\phi^{1}_{i}\}_{i\in\N}$ and $S_{2}=\{\phi^{2}_{i}\}_{i\in\N}$ into $\big(\C^G(M), \{0\}\big)$, \emph{$S_{1}$ is homotopic to $S_{2}$} if there exists $ \{\de_{i}>0\}_{i\in\N}$, such that
\begin{itemize}
\setlength{\itemindent}{1em}
\item $\phi^{1}_{i}$ is $1$-homotopic to $\phi^{2}_{i}$ in $\big(\C^G(M), \{0\}\big)$ with fineness $\de_{i}$;
\item $\lim_{i\rightarrow \infty}\de_{i}=0$.
\end{itemize}
\end{definition}

It is easy to see that the homotopic relation is an equivalence relation on the space of $(1, \M)$-homotopy sequences of mappings into $\big(\C^G(M), \{0\}\big)$. An equivalence class is named as a \emph{$(1, \M)$-homotopy class of mappings into $\big(\C^G(M), \{0\}\big)$}. The set of all such equivalence classes is denoted by $\pi^{\#}_{1}\big(\C^G(M, \M), \{0\}\big)$.

\subsection{Min-max construction.}

\begin{definition}
(Min-max definition) Given $\Pi\in\pi^{\#}_{1}\big(\C^G(M, \M), \{0\}\big)$, define  $\bL^c: \Pi\rightarrow\R^{+}$ as a function given by:
\[ \bL^c(S)=\bL^c(\{\phi_{i}\}_{i\in\N})=\limsup_{i\rightarrow\infty}\max\big\{\Ac\big(\phi_{i}(x)\big):\ x \textrm{ lies in the domain of $\phi_{i}$}\big\}. \]
The \emph{$\mathcal{A}^c$-min-max value of $\Pi$} is defined as
\begin{equation}\label{width}
\bL^c(\Pi)=\inf\{\bL^c(S):\ S\in\Pi\}.
\end{equation}
A sequence $S=\{\phi_i\}_{i\in\N}\in\Pi$ is called a \emph{critical sequence}  if $\bL^c(S)=\bL^c(\Pi)$.
The {\em image set} of $S$ is the compact subset $K(S)\subset\mathcal{V}^G_n(M^{n+1})$ defined by
 $$K(S)=\{V=\lim_{j\to \infty}|\partial \phi_{i_{j}}(x_{j})|:\ \textrm{$x_{j}$ lies in the domain of $\phi_{i_{j}}$}\}.$$
The \emph{critical set} of $S$ is the subset $C(S)\subset K(S)\subset\mathcal{V}^G_n(M)$ defined by
\begin{equation}\label{critical set}
 C(S)=\{V=\lim_{j\rightarrow\infty}|\partial\phi_{i_j}(x_j)|:\ \text{with} \lim_{j\to\infty} \Ac(\phi_{i_j}(x_j))=\bL^c(S)\}.
\end{equation}
\end{definition}

%%%%

Note that by a diagonal sequence argument as in \cite[4.1(4)]{pitts}, we immediately have:
\begin{lemma}
\label{L:Existence of critical sequence}
Given any $\Pi\in\pi^{\#}_{1}\big(\C^G(M, \M), \{0\}\big)$, there exists a critical sequence $S\in \Pi$.
\end{lemma}

\subsection{Existence of nontrivial $G$-sweepouts}
\label{SS:Existence of nontrivial sweepouts}

\begin{theorem}
\label{T:existence of nontrivial G-sweepouts}

There exists $\Pi\in\pi^{\#}_{1}\big(\C^G(M, \M), \{0\}\big)$, such that for any $c>0$, we have $\bL^c(\Pi)>0$.
\end{theorem}

\begin{proof}
 First, using Proposition \ref{G-current compactness}, \ref{G-varifold compactness}, \ref{restrict}, \ref{slice} and \ref{isoperimetric}, the discretization/interpolation theorem given by Xin Zhou (\cite[Theorem 5.1]{zhou}) can be adapted to a $G$-invariant version. Secondly, take a $G$-equivariant Morse function $\phi : M\rightarrow[0,1]$, and consider the sub-level sets $\Phi: [0,1]\rightarrow\C^G(M)$, given by $\Phi(t)=\{x\in M : \phi(x)<t \}$. Then the map $\partial{\Phi}: [0,1]\rightarrow ({\mathcal Z}_n^G(M),\mathcal{F})$ is continuous and $\M(\partial\Phi(t))$ is continuous in $t\in[0,1]$. Now by the $G$-invariant discretization/interpolation theorem, $\Phi$ can be discretized to a $(1, \M)$-homotopy sequences $S=\{\phi_{i}\}_{i\in\N}$ where $\phi_{i}: I(1, k_i)_{0}\rightarrow\big(\C^G(M), \{0\}\big)$. By taking $\Pi=[S]$, then the proof for  $\Pi\in\pi^{\#}_{1}\big(\C^G(M, \M), \{0\}\big)$ with $\bL^c(\Pi)>0$ proceeds as in the CMC setting \cite[Theorem 3.9]{zhou and zhu}.
\end{proof}

\medskip
\section{Tightening}\label{Tightening1}

In this section, we recall the tightening map adapted to the $\Ac$ functional constructed in \cite{zhou and zhu}. We will prove that after applying the modified tightening map to a critical sequence, every element in the critical set has $c$-bounded first variation.

\subsection{Review of constructions in \cite[\S 4]{zhou and zhu} }
\label{SS:Vector field and diffeomorphisms}{~}

We first recall several key ingredients obtained in \cite[\S 4]{zhou and zhu} for the tightening
process. Given $L>0$ and $c>0$, we will use the following notations:
\begin{itemize}
\setlength{\itemsep}{5pt}
	\item $A^L = \{V\in\mathcal{V}^G_n(M): \|V\|(M)\leq2L \}$;
	\item $A^{c}_{\infty}= \{V\in A^L: |\de V(X)|\leq c \int_M|X|d\mu_V,  \forall X\in\mathfrak{X}^G(M) \}$;
	\item $A_j =\{V\in A^L: \frac{1}{2^j}\leq {\bf F}(V,A^{c}_{\infty})\leq \frac{1}{2^{j-1}}\}$, $j\in \N$;
    \item $\gamma(V)={\bf F}(V,A^{c}_{\infty})$, \text{for any} $V\in A^L$;
    \item $\Phi_X:\R\times M\rightarrow M $ denotes the one parameter group of $G$-equivariant diffeomorphisms generated by $X\in\mathfrak{X}^G(M)$.
\end{itemize}

For any $V\in A^{c}_{\infty}$, $V$ has $(G,c)$-bounded first variation.
By Proposition \ref{G-varifold compactness} and the fact that $(G,c)$-bounded first variation is a closed condition, we have $A^L$, $A_j$, $A^{c}_{\infty}$ are compact subsets of $\mathcal{V}^G_n(M)$ under the ${\bf F}$-metric for all $j\in \N$, $A_j$, $A^{c}_{\infty}$.

We will need the following lemma which can be used to construct a continuous map from $A^L$ to $\mathfrak{X}^G(M)$.
\begin{lemma}
\label{G-invariant vector field}
For any $V\in A_j$, there exists $X_V\in \mathfrak{X}^G(M)$, such that
\begin{equation}\label{vector}
\|X_V\|_{C^1(M)}\leq 1, \quad \de V(X_V)+c \int_M|X_V|d\mu_V\leq -c_j<0
\end{equation}
where $c_j$ depends only on $j$.
\end{lemma}
\begin{proof}
This follows from a contradiction argument using the compactness of $A_j$.
\end{proof}
The following lemma is a straightforward consequence of Lemma \ref{G-invariant vector field} and the construction in \cite[\S 4.2]{zhou and zhu}.
\begin{lemma}
\label{a map from varifold to vector field }
There exists a continuous map $X: A^L \rightarrow \mathfrak{X}^G(M)$ with respect to the $C^1$ topology on $\mathfrak{X}^G(M)$, so that $X(V)$ satisfies (\ref{vector}) for all $V\in A^L$.
\end{lemma}

 Our goal is to show that given $\Om\in \C^G(M)$ with $|\partial\Om| \in A^L\setminus A^{c}_{\infty}$, then $\Ac(\Om)$ can be deformed down by a fixed amount depending only on $\gamma(|\partial\Om|)={\bf F}(|\partial\Om|,A^{c}_{\infty})$. The following results have been obtained as in \cite[\S 4.3]{zhou and zhu}:

\begin{proposition}\label{G-equivariant Deformation}
($G$-equivariant Deformation)
\begin{enumerate}[(1)]
\setlength{\itemsep}{5pt}
  \item There are two continuous functions $g:\R^{+}\rightarrow \R^{+}$ and $\rho:\R^{+}\rightarrow \R^{+}$, such that $\rho(0)=0$ and
   \begin{equation}
   \quad \de W(X(V))+c \int_M|X(V)|d\mu_W\leq -g(\gamma(V)),
   \end{equation}
   if $W\in A^L$ and ${\bf F}(W,V)\leq \rho(\gamma(V))$;
  \item There exists a continuous time function $T:[0,+\infty)\rightarrow[0,+\infty)$, such that

      \begin{itemize}
      \setlength{\itemsep}{5pt}
      \item $\lim_{t\rightarrow 0}T(t)=0$, and $T(t)>0$ if $t\neq0$;
      \item For any $V\in A^L$, denote $\Phi_V=\Phi_{X(V)}$, then for all $0\leq t\leq T(\gamma(V))$, we have
      \begin{equation}
      {\bf F}((\Phi_{V}(t))_{\#}V,V)< \rho(\gamma(V));
      \end{equation}
      \end{itemize}
  \item For any $V\in A^L\setminus A^{c}_{\infty}$, define $\Psi_V(t,\cdot)=\Phi_V(T(\gamma(V))t,\cdot)$ and $V_t=(\Psi_{V}(t))_{\#}V$, we have
      \begin{itemize}
      \setlength{\itemsep}{5pt}
      \item ${\bf F}(V_t,V)< \rho(\gamma(V))$ for all $t\in[0,1]$;
      \item The map $(t,V)\rightarrow V_t$ is continuous under the ${\bf F}$-metric;
      \item The flow $\Psi_V(t,\cdot)=\Phi_V(T(\gamma(V))t,\cdot)$ is generated by the vector field
      \begin{equation}\label{new vector field}
      \tilde{X}(V)=T(\gamma(V))X(V).
      \end{equation}
      \end{itemize}
\end{enumerate}
\end{proposition}

%Note that  $X(V)$ is a $G$-vector field, hence so is $\tilde{X}(V)$.

%The goal of the section is to show that given $\Om\in \C^G(M)$ with $|\partial\Om| \in A^L\setminus A^{c}_{\infty}$, we can deform
%$\Om$ by $\Psi_{|\partial\Om|}(t)$ to get a 1-parameter family $\Om_t=(\Psi_{|\partial\Om|}(t))_{\#}(\Om)\in \C^G(M)$, such that
% $\Ac(\Om_t)$ for some $0\leq t\leq 1$ can be deformed down by a fixed amount depending only on
% $\gamma(|\partial\Om|)={\bf F}(|\partial\Om|,A^{c}_{\infty})$.
% denote $V=|\partial\Om|$,

When $\Om\in \C^G(M)$ with $|\partial\Om| \in A^L\setminus A^{c}_{\infty}$,  Proposition \ref{G-equivariant Deformation}(1) implies
 $$\quad \de V(X(|\partial\Om|))+c \int_M|X(|\partial\Om|)|d\mu_V \leq -g(\gamma(|\partial\Om|)),$$
for any $V\in A^L$ and ${\bf F}(V,|\partial\Om|)\leq \rho(\gamma(|\partial\Om|))$. Noting $\tilde{X}(|\partial\Om|)\in \mathfrak{X}^G(M)$, we can deform $\Om$ by $\Psi_{|\partial\Om|}(t)$ to get a 1-parameter family $\Om_t=(\Psi_{|\partial\Om|}(t))_{\#}(\Om)\in \C^G(M)$.  By (\ref{E: 1st variation for Ac}) and
  Proposition \ref{G-equivariant Deformation},  we have

  \begin{align*}
   \Ac(\Om_1)-\Ac(\Om) &\leq \int_{0}^{T(\gamma(|\partial\Om|))}[\de\Ac|_{\Om_t}](X(|\partial\Om|))dt \\
                       &\leq -T(\gamma(|\partial\Om|))g(\gamma(|\partial\Om|))=-L(\gamma(|\partial\Om|))<0.
   \end{align*}
where $L(\gamma(V))=T(\gamma(V))g(\gamma(V))$, with $L(0)=0$ and $L(\gamma(V))>0$ if $\gamma(V)>0$.

\subsection{Deforming sweepouts by the tightening map}
\label{pull-tight}{~}

We now apply our tightening map to the critical sequence provided by
Lemma \ref{L:Existence of critical sequence}. As in the usual min-max theory for the CMC setting, we confirm the existence of a critical sequence $S$ with each varifold in $C(S)$ has $c$-bounded first variation. Indeed, the proof proceeds essentially unchanged, with $X\in\mathfrak{X}(M)$ replaced by $X\in \mathfrak{X}^G(M)$.

\begin{proposition}\label{Tightening}(Tightening)
Let $\Pi\in\pi^{\#}_{1}\big(\C^G(M, \M), \{0\}\big)$ with $\bL^c(\Pi)>0$. For any critical sequence $S^\ast$ for $\Pi$, there exists another critical sequence $S$ for $\Pi$ such that $C(S)\subset C(S^\ast)$ and each $V\in C(S)$ has c-bounded first variation.
\end{proposition}

\begin{proof}
 Take $S^\ast=\{\varphi^\ast_i\}_{i\in\N}$ , where $\varphi^\ast_i : I(1, k_{i})_{0}\rightarrow\big(\C^G(M), \{0\}\big)$, and $\varphi^\ast_i$ is $(1,\M)$-homotopic to $\varphi^\ast_{i+1}$ in $\big(\C^G(M), \{0\}\big)$ with fineness $\delta_i\searrow 0$. Let $\Xi_i : I(1, k_{i})_{0}\times[0,1]\rightarrow\C^G(M)$ be defined as
 $$\Xi_i(x,t)=\Psi_{|\partial\varphi^\ast_i(x)|}(t)(\varphi^\ast_i(x)),$$
 where $\Psi_{|\partial\varphi^\ast_i(x)|}$ is the tightening map defined in Proposition \ref{G-equivariant Deformation}.

 Denote $\varphi^t_i(x)=\Xi_i(x,t)$. The idea is to extend $\varphi^1_i(x)$ to a continuous map $\bar{Q}_i(x)$ defined on $[0,1]$, with respect to the ${\bf F}$-metric, and then apply the discretization-interpolation theorem in \cite[Theorem 5.1]{zhou} to $\bar{Q}_i(x)$. To be exact, let $\tilde{X}_i(x)=(1-x)\tilde{X}_i(0)+x\tilde{X}_i(1)$ and $\bar{Q}_i(x)=(\Phi_{\tilde{X}_i(x)}(1))_{\#}\varphi^\ast_i(0)$, where $\tilde{X}_i(0)=\tilde{X}(|\partial\varphi^\ast_i(0)|)$ and $\tilde{X}_i(1)=\tilde{X}(|\partial\varphi^\ast_i(1)|)$ are defined in Proposition \ref{G-equivariant Deformation} (\ref{new vector field}). Note $\tilde{X}_i(x)\in\mathfrak{X}^G(M)$. By Proposition \ref{Prop-G-push forward}, Lemma \ref{a map from varifold to vector field } and Proposition \ref{G-equivariant Deformation}, both $\bar{Q}_i(x):[0,1]\rightarrow(\C^G(M),\M)$ and $\partial\bar{Q}_i(x):[0,1]\rightarrow({\mathcal Z}_n^G(M),{\bf F})$ are continuous maps.
  Thus we can apply the discretization-interpolation result in \cite[Theorem 5.1]{zhou} to get the desired sequence $S=\{\varphi_i\}_{i\in\N}$ which is $(1,\M)$-homotopic to $S^\ast=\{\varphi^\ast_i\}_{i\in\N}$. The proof proceeds essentially unchanged, just as in \cite[Proposition 4.4 and Appendix B]{zhou and zhu}.
\end{proof}

\medskip
\section{$(G,c)$-almost minimizing varifolds}\label{(G,c)-Almost Minimizing}
In this section, we introduce the notion of $(G,c)$-almost minimizing varifolds, and show the existence of such a varifold from min-max construction.
\begin{definition}\label{$(G,c)$-almost minimizing varifolds}
($(G,c)$-almost minimizing varifolds) Let $\nu$ be the $\F$-norm or $\M$-norm, or the ${\bf F}$-metric. For any given $\ep$,$\de>0$ and an open $G$-set $U\subset M$, we define $\mathfrak{a}^{c,G}_n(U;\ep,\de;\nu)$ to be the set of all $\Omega\in\C^G(M)$ such that if  $\Omega=\Omega_0,\Omega_1,\Omega_2,\cdots,\Omega_m\in\C^G(M)$ is a sequence with:
\begin{itemize}
\item[(i)] $\spt(\Om_i-\Om)\subset U$;
\item[(ii)] $\nu(\partial\Om_{i+1} - \partial\Om_{i})\leq \de$;
\item[(iii)] $\Ac(\Om_i)\leq \Ac(\Om)+\de$, for $i=1, \cdots, m$.
\end{itemize}
then $\Ac(\Om_m)\geq\Ac(\Om)-\ep$.

We say that a varifold $V\in\mathcal{V}_n^G(M)$ is $(G,c)$-almost minimizing in $U$ if there exist sequences $\ep_i\rightarrow0$, $\de_i\rightarrow0$, and $\Omega_i\in\mathfrak{a}^{c,G}_n(U;\ep_i,\de_i;\F)$, such that ${\bf F}(|\partial\Omega_i|,V)\leq\ep_i$.
\end{definition}
A simple consequence of the definition is that $(G,c)$-almost minimizing in $U$ implies $c$-bounded first variation in $U$:

\begin{lemma}\label{(G,c)-almost minimizing implies c-bounded first variation}
Let $V\in\mathcal{V}_n^G(M)$ be $(G,c)$-almost minimizing in $U$, then $V$ has $c$-bounded first variation in $U$.
\end{lemma}

\begin{proof}
This lemma follows similarly to \cite[Lemma 5.2]{zhou and zhu}. We include the details for completeness.
By Lemma \ref{$c$-bounded first variation}, $V$ has $c$-bounded first variation in $U$ if and only if $V$ has $(G,c)$-bounded first variation in $U$. Suppose by contradiction that $V$ does not have $(G,c)$-bounded first variation in $U$, then there exist $\ep_0>0$ and a smooth vector field $X\in\mathfrak{X}^G(U)$ compactly supported in $U$, such that
$$\int_{G_n(M)} \Div_{S}X(x)dV(x,S)\leq -(c+\ep_0) \int_M|X|d\mu_V$$
By the continuity and the first variation formula (\ref{E: 1st variation for Ac}) for $\Ac$, we can find $\ep_1>0$ small enough depending only on $\ep_0$, $V$, and $X$, such that if $\Omega\in\C^G(M)$ with ${\bf F}(|\partial\Omega|,V)<2\ep_1$, then
$$\de\Ac|_{\Om}(X)\leq \int_{\partial\Omega} \Div_{\partial\Omega}Xd\mu_{\partial\Omega}+c\int_{\partial\Omega}|X|d\mu_{\partial\Omega}\leq -\frac{\ep_0}{2}\int_M|X|d\mu_V<0.$$
If ${\bf F}(|\partial\Omega|,V)<\ep_1$, we then consider the deformation of $\Omega$ along the flow $\{\Phi_{X}(t): t\in [0,T]\}$ generated by the $G$-vector field $X$ for a uniform short time $T>0$.
Just as in Proposition \ref{G-equivariant Deformation}, we can get a $1$-parameter family $\{\Omega_t\in\C^G(M): t\in[0,T]\}$, such that $t\rightarrow \partial\Omega_t$ is continuous under the $\bf F$-metric, with ${\rm spt}(\Omega_t-\Omega)\subset U$, ${\bf F}(|\partial\Omega_t|,V)<2\ep_1$, and $\Ac(\Om_t)\leq \Ac(\Om_0)=\Ac(\Om)$ for $t\in [0,T]$, but with $\Ac(\Om_T)\leq \Ac(\Om)-\ep_2$ for some $\ep_2>0$ depending only on $\ep_0$, $\ep_1$, $V$, $X$.

This implies that if $\Omega\in\C^G(M)$ and ${\bf F}(|\partial\Omega|,V)<\ep=\frac{1}{2}\min\{\ep_1,\ep_2\}$, then $\Omega\notin\mathfrak{a}^{c,G}_n(U;\ep,\de;\F)$ for any $\de>0.$ This contradicts to the assumption that $V\in\mathcal{V}_n^G(M)$ is $(G,c)$-almost minimizing in $U$.
\end{proof}

\begin{definition}\label{(G,c)-almost minimizing in annuli}
 A varifold $V\in\mathcal{V}_n^G(M)$ is said to be {\em $(G,c)$-almost minimizing in small annuli}, if for each $p\in M$, there exists $r=r(p)>0$ such that  $V$ is $(G,c)$-almost minimizing in $\an(p,s,t)$ for any $\an(p,s,t)\in\ann_{r}(p)$.		
\end{definition}

Note the \emph{fineness} of discrete homotopy sequence of mappings $\mf(\phi)$ is defined under the $\M$-norm, while the $(G,c)$-almost minimizing varifolds are defined under the $\F$-metric. We need the following equivalence results of almost minimizing concepts using different topologies.

\begin{theorem}\label{Thm: equivalence-a.m.v}
	Let $V\in\mathcal{V}^G_n(M)$.
	The following statements satisfy $(a) \Rightarrow (b) \Rightarrow (c)\Rightarrow(d)$:
	\begin{itemize}
		\item[(a)] $V$ is $(G,c)$-almost minimizing in $U$;
		\item[(b)] for any $\epsilon>0$, there exist $\delta>0$ and $\Om\in \mathfrak{a}^{c,G}_n(U;\epsilon,\delta; {\bf F} )$ such that ${\bf F}(V,|\partial \Om|)<\epsilon$;
		\item[(c)] for any $\epsilon>0$, there exist $\delta>0$ and $\Om\in \mathfrak{a}^{c,G}_n(U;\epsilon,\delta; {\bf M} )$ such that ${\bf F}(V,|\partial \Om|)<\epsilon$;
		\item[(d)] $V$ is $(G,c)$-almost minimizing in $W$ for any relative open $G$-subset $W\subset\subset U$ with ${\rm Clos}(W)\subset M^{reg}$, where $M^{reg}$ is the union of
           all principal orbits.
	\end{itemize}
\end{theorem}
\begin{proof}
	It's clear that $(a) \Rightarrow (b) \Rightarrow (c)$ by definitions.
	
	To show $(c)\Rightarrow (d)$, we first need an equivariant version of boundary type interpolation lemma (i.e. using boundaries of $G$-invariant Caccioppoli sets, rather than integral cycles $\mathcal{Z}_n(M)$, as interpolating sequences in \cite[Lemma 3.8]{pitts}).
	In fact, \cite[Appendix B]{WTR} has shown how to get interpolating sequences consisting integral $G$-cycles $T\in\mathcal{Z}^G_n(M)$.
	In the arguments of \cite[Appendix B]{WTR}, the assumption that every orbit in ${\rm Clos}(W)$ is principal plays a necessary role in getting a positive lower bound of ${\rm Inj}(G\cdot p)$ the injective radius of orbits and the Weyl's tube formula \cite[(18)]{WTR}.
	Additionally, \cite[Proposition 5.5]{zhou} has essentially the same idea as \cite[Lemma 3.7, 3.8]{pitts} except using Caccioppoli sets, rather than integral cycles $\mathcal{Z}_n(M)$, as interpolating sequences.
	Thus, we can combine the arguments in \cite[Appendix B]{WTR} and \cite[Section 5.1]{zhou} to deduce an equivariant version of \cite[Proposition 5.3]{zhou}.
	 With the equivariant version of \cite[Proposition 5.3]{zhou} ($G$-boundary type interpolation lemma), the proof of \cite[Lemma A.1]{zhou and zhu} would carry over for $G$-invariant Caccioppoli sets.
\end{proof}

\begin{corollary}\label{Cor:equivalence (G,c)-almost minimizing in annuli}
	Let $V\in\mathcal{V}^G_n(M)$.
	Suppose every non-principal orbit is isolated.
	Then $V$ is $(G,c)$-almost minimizing in annuli if and only if for any $p\in M$ there exists $r(p)>0$ so that for any $0<s<t<r(p)$ and $\epsilon>0$ there exist $\delta>0$ and $T\in \mathfrak{a}^{c,G}_n(\an(p,s,t);\epsilon,\delta; {\bf M} )$ with ${\bf F}(V,|T|)<\epsilon$.
\end{corollary}
\begin{proof}
	The necessity part comes from Theorem \ref{Thm: equivalence-a.m.v} (a)$\Rightarrow$(c).
	As for the sufficiency part, one notices that if $G\cdot p$ is a principal orbit in $M$ then there naturally exists a $G$-neighborhood of $G\cdot p$ in $M^{reg}$ by the openness of $M^{reg}$.
	Additionally, since every non-principal orbit is isolated, if $G\cdot p$ is non-principal then there exists a $G$-neighborhood $B_{r}^G(p)$ of $G\cdot p$ so that $B^G_r(p)\setminus G\cdot p\subset M^{reg}$.
	Hence, there always exists ${r}(p)>0$ so that ${\rm Clos}(\an(p,s,t))\subset M^{reg}$ for any $0<s<t<r(p)$.
	Therefore, the sufficiency comes from Theorem \ref{Thm: equivalence-a.m.v} (c)$\Rightarrow$(d).
\end{proof}

In above corollary, we use the assumption that every non-principal orbit is isolated to guarantee that ${\rm Clos}(\an(p,s,t))\subset M^{reg}$ for any $p\in M$ and $0<s<t$ small enough. To weaken the constraints on non-principal orbits, we can consider the following notations and definitions just as in \cite[Section 5]{WTR}.
Firstly, we denote
\begin{equation}\label{Eq:Pp}
	P_p := \left\{ \begin{array}{ll}
 		p & \textrm{if $p\in M^{reg}$}\\
 		M_p & \textrm{if $p\in M\setminus M^{reg}$},
  	\end{array} \right.
\end{equation}
where $M_p$ is the connected component of $M\setminus M^{reg}$ containing $p$.
Hence, for any $p\in M$, there exists $r_1=r_1(p)>0$ so that
\begin{equation}\label{Eq:regular annuli}
	{\rm Clos}(\an(P_p,s,t))\subset M^{reg},~\forall \an (P_p,s,t)\in\ann_{r_1}(P_p),
\end{equation}
by regarding $M\setminus M^{reg}$ as a whole.
Since the closure of such annulus is contained in $M^{reg}$, $G$-annulus $\an(P_p,s,t)$ are called  {\em regular annulus}.
In particular, it's easy to check that $G\cdot P_p \equiv G\cdot p$ when every non-principal orbit is isolated.

\begin{definition}\label{Def:a.m. in regular annuli}
	A $G$-varifold $V \in \mathcal{V}^G_n(M)$ is said to be {\it $(G,c)$-almost minimizing in regular annuli} if
	for each $p \in M$, there exists $r=r(p)\in (0,r_1(p))$ such that $V$ is $(G,c)$-almost minimizing in $\an (P_p,s,t)$ for any $\an (P_p,s,t)\in\ann_r(P_p) $, where $r_1(p)$ is given by (\ref{Eq:regular annuli}).
\end{definition}

Combining (\ref{Eq:regular annuli}) with Theorem \ref{Thm: equivalence-a.m.v}, it's easy to get the following corollary:
\begin{corollary}\label{Cor:equivalence-a.m.v in regular annuli}
	Let $V\in\mathcal{V}^G_n(M)$.
	Then $V$ is {\it $(G,c)$-almost minimizing in regular annuli} if and only if for any $p\in M$ there exists $r=r(p)\in (0,r_1(p))$ so that for any $0<s<t<r(p)$ and $\epsilon>0$ there exists $\delta>0$ and $T\in \mathfrak{a}^{c,G}_n(\an(P_p,s,t);\epsilon,\delta; {\bf M} )$ with ${\bf F}(V,|T|)<\epsilon$.
\end{corollary}
\begin{proof}
	The necessity part comes from Theorem \ref{Thm: equivalence-a.m.v} (a)$\Rightarrow$(c).
	As for the sufficiency part, one notices that ${\rm Clos}(\an(P_p,s,t))\subset M^{reg}$ for any $0<s<t<r_1(p)$.
	Therefore, the sufficiency comes from Theorem \ref{Thm: equivalence-a.m.v} (c)$\Rightarrow$(d).
\end{proof}

\begin{theorem}\label{existence (G,c)-almost minimizing in annuli}
(Existence of $(G,c)$-almost minimizing varifold)\\
 Let $\Pi\in\pi^{\#}_{1}\big(\C^G(M, \M), \{0\}\big)$ with $\bL^c(\Pi)>0$. There exists a nontrivial varifold $V\in\mathcal{V}_n^G(M)$ such that
\begin{itemize}
\item[(i)] $V\in C(S)$ for some critical sequence $S$ of  $\Pi$;
\item[(ii)] $V$ has $c$-bounded first variation in $M$;
\item[(iii)] $V$ is {\it $(G,c)$-almost minimizing in regular annuli}.
\end{itemize}
\end{theorem}

\begin{proof}
It is easy to see that (i), (ii) come from Theorem \ref{T:existence of nontrivial G-sweepouts} and Proposition \ref{Tightening}. The proof of \cite[Theorem 5.6(iii)]{zhou and zhu} would carry over under $G$-invariant restrictions with Theorem \ref{Thm: equivalence-a.m.v} and \cite[Theorem 5.10]{WTR} in places of \cite[Theorem 5.5]{WTR} and \cite[Theorem 4.10]{pitts}.
\end{proof}

\section{Regularity for $(G,c)$-min-max varifold}\label{Regularity for $(G,c)$-min-max varifold}
In this section, we show Regularity for $(G,c)$-min-max varifold.
First of all, let's begin with the existence of $(G,c)$-replacement for $(G,c)$-almost minimizing varifolds.

\medskip
\subsection{Good $(G,c)$-replacement}~

\begin{lemma}[A constrained minimization problem-I]
\label{L:minimisation}
Given $\epsilon, \delta>0$, an open $G$-set $U\subset M$ and any $\Omega \in \mathfrak{a}^{c,G}_n(U;\ep,\de;\F)$, fix a compact $G$-subset $K\subset U$.
Let $\mathcal{C}^G_\Omega$ be the set of all $\La\in\C^G(M)$ such that there exists a sequence $\Om=\Om_0, \Om_1, \cdots, \Om_m=\La$ in $\C^G(M)$ satisfying:
 \begin{itemize}
\item[(a)] $\spt(\Om_i-\Om)\subset K$;
\item[(b)] $\F(\partial\Om_i - \partial\Om_{i+1})\leq \de$;
\item[(c)] $\Ac(\Om_i)\leq \Ac(\Om)+\de$, for $i=1, \cdots, m$.
\end{itemize}
Then there exists $\Omega^* \in \C^G(M)$ such that:
\begin{itemize}
\item[(i)] $\Om^* \in \C^G_\Om$, and \[ \Ac(\Om^*)=\inf\{\Ac(\La):\ \La\in\C^G_\Om\},\]
\item[(ii)] $\Om^*$ is locally $(G,\Ac)$-minimizing in ${\rm Int}(K)$,
\item[(iii)] $\Om^*\in \mathfrak{a}^{c,G}_n(U;\ep,\de;\F)$.
\end{itemize}
\end{lemma}

Here we say $\Om^*$ is locally $(G,\Ac)$-minimizing in ${\rm Int}(K)$ means that for any $p\in {\rm Int}(K)$ there exists $\rho>0$ so that
$$\Ac(\Om^*)\leq\Ac(\La),$$
for any $\La\in \C^G(M)$ with $ \mbox{spt}(\La-\Om^*)\subset \sB^G_\rho(p)\subset {\rm Int}(K)$.

\begin{proof}
	Let $\{\La_j\}\subset\C_{\Om}^G$ be a minimizing sequence such that
	\[ \lim_{j \to \infty} \Ac(\La_j) = \inf\{\Ac(\La):\ \La\in\C^G_{\Om}\}.\]
	Since $\mbox{spt}(\Om-\La_j)\subset K$ and $\Ac(\La_j)\leq \Ac(\Om)+\delta$, we can apply the compactness theorem \cite[Theorem 6.3]{simon} and Proposition \ref{G-current compactness} to get a subsequence $\partial\La_j$ (without changing notations) converges weakly to $\partial\Om^*$ for some $\Om^*\in\C^G(M)$ with $\spt(\Om^*-\Om)\subset K$.
	Hence by the weakly convergence, we have
	\begin{itemize}
		\item ${\bf M}(\partial \Om^*)\leq \liminf_{j\rightarrow\infty} {\bf M}(\partial \La_j)$,
		\item $\mH^{n+1}(\Om^*)=\lim_{j\to\infty}\mH^{n+1}(\La_j)$.
	\end{itemize}
	This implies that $\Ac(\Om^*) \leq  \inf\{\Ac(\La) : \La\in\C_{\Om}\} \leq \Ac(\Om)+\delta$.
	Moreover, for $j$ large enough, we have $\mathcal{F}(\partial\La_j-\partial\Om^*)< \delta$. Consider the sequence which make $\La_j$ in $\mathcal{C}_\Om^G$ and add one more element $\Om^*$ into its end. It's clear that the new sequence satisfies all requirements in the definition of $\mathcal{C}_\Om^G$.  Thus we get $\Om^*\in \mathcal{C}^G_\Om$ and $\Ac(\Om^*)= \mbox{inf}\{ \Ac(\La): \La\in \mathcal{C}_\Om^G \}$.

	Now for any $p\in {\rm Int}(K)$, we first choose $\rho>0$ small enough such that $\sB^G_\rho(p)\subset {\rm Int}(K)$.
		Since $|\partial\Om^*|$ is rectifiable, one can take $\rho$ even smaller so that
		\begin{itemize}
			\item $c\mH^{n+1}(\sB^G_\rho(p))\leq \delta/4 $,
			\item $||\partial\Om^*||(\sB^G_r(p)) \leq \delta/4$.
		\end{itemize}
	If there exists $\La\in \C^G(M)$ with $ \mbox{spt}(\La-\Om^*)\subset \sB^G_\rho(p)$ such that $\Ac(\Om^*)>\Ac(\La)$.
	Then
	\begin{eqnarray*}
		||\partial\La||(\sB^G_\rho(p)) & < & ||\partial\Om^*||(\sB^G_\rho(p))-c\mH^{n+1}(\Om^*)+c\mH^{n+1}(\La)
		\\
		&\leq & ||\partial\Om^*||(\sB^G_\rho(p))+2c\mH^{n+1}(\sB^G_\rho(p))
		\\
		&\leq & 3\delta/4.
	\end{eqnarray*}
	This implies that
		\begin{eqnarray*}
			\mathcal{F}(\partial\La-\partial\Om^*) &\leq & {\bf M}(\partial\La-\partial\Om^*)
			\\
			&=& ||\partial\La-\partial\Om^*||(\sB^G_\rho(p))
			\\
			&\leq & (||\partial\La||+||\partial\Om^*||)(\sB^G_\rho(p))< \delta.
		\end{eqnarray*}
	Similar to before, we can add $\La$ to the end of the finite sequence which makes $\Om^*$ in $\mathcal{C}^G_\Om $.
	Hence, $\La\in \mathcal{C}_\Om^G$ and $\Ac(\Om^*)>\Ac(\La)$ which is a contradiction to the choice of $\Om^*$.

	As for the last statement, suppose there is a sequence $\Om^*=\Om^*_0,\Om^*_1,\dots,\Om^*_q$ forms a $G$-invariant $(\epsilon,\delta)$-deformation of $\Om^*$ in $U$.
	Consider the finite sequence $\{\Om_i\}_{i=1}^{m}$ which make $\Om^*=\Om_m$ in $\mathcal{C}_\Om^G$ and insert $\{\Om^*_i\}_{i=1}^q$ at its end.
	According to the Definition  \ref{$(G,c)$-almost minimizing varifolds}, we get a $G$-invariant $(\epsilon,\delta)$-deformation $\{\Om_1,\dots,\Om_m,\Om_0^*,\dots,\Om_q^*\}$ of $\Om=\Om_1 $ which is a contradiction with the choice of $\Om\in \mathfrak{a}^{c,G}_n(U;\ep,\de;\F)$.
\end{proof}

\medskip
Noting that $\Om^*$ in this lemma can only minimize the $\Ac$ function under $G$-invariant locally variations.
But following an averaging procedure, one can get the next lemma which shows that $\Om^*$ is locally $\Ac$-minimizing and thus have good regularity.

\begin{lemma}[A constrained minimization problem-II]
\label{L:minimisation-II}
	The $G$-Caccioppoli set $\Om^* \in \C^G(M)$ obtained in Lemma \ref{L:minimisation} is locally $\Ac$-minimizing in $\interior(K)$.
\end{lemma}

\begin{proof}
	For any $p\in \interior(K)$, let $\rho>0$ such that $\sB^G_\rho(p)\subset \interior(K)$ and
	\begin{itemize}
		\item $c\mH^{n+1}(\sB^G_\rho(p))\leq \delta/4 $,
		\item $||\partial\Om^*||(\sB^G_{\rho}(p))\leq \delta/4$.
	\end{itemize}
	Then $\Om^*$ is $(G,\Ac)$-minimizing in $\sB^G_\rho(p)$ by the proof of Lemma \ref{L:minimisation}.

	Let $B=\sB^G_{\rho}(p)$.
	We now going to show that $\Om^*$ is $\Ac$-minimizing in $B$.
	Suppose $\La^*\in  \C(M)$ such that $\mbox{spt}(\Om^*-\La^*)\subset B$ and $\Ac(\La^*)\leq \Ac(\Om^*)$.
	It's sufficient to show that $\Ac(\La^*)= \Ac(\Om^*)$.

	Define a $G$-invariant function on $M$ as:
	$$ f(x)=\int_G {\rm 1}_{\Clos(\La^*)}(g\cdot x)~d\mu(g). $$
	Since $\Clos(\La^*)$ is closed, function $ {\rm 1}_{\Clos(\La^*)}$ is upper-semicontinuous. By Fatou Lemma, $f$ is also upper-semicontinuous too.
	Hence, $$\La_\lambda = f^{-1}[\lambda,1 ]= M\backslash f^{-1}[0,\lambda )$$
	are $G$-invariant closed sets in $M$.
	
	Define then $E_f = f\cdot [[M]]$, where $[[M]]$ is the integral current induced by $M$.
	For any $n$-form $\omega$ on $M$, we have
	\begin{eqnarray}\label{Eq-boundary-Ef}
		\partial E_f(\omega) &=& \int_{M}\int_G \langle d\omega, \xi\rangle ~{\rm 1}_{\Clos(\La^*)}(g\cdot x)~d\mu(g)d\mathcal{H}^{n+1}(x)
		\\
		&=& \int_G\int_{M} \langle d\omega, \xi\rangle ~{\rm 1}_{g^{-1}(\Clos(\La^*))}(x)~d\mathcal{H}^{n+1}(x)d\mu(g)\nonumber
		\\
		&=& \int_G \partial ((g^{-1})_\# \La^*)(\omega) ~d\mu(g).\nonumber
	\end{eqnarray}
	Hence, using the lower semi-continuity of mass and the fact that $G$ act as isometries on $M$ (mass is invariant under $g_\#$) we have
	\begin{equation}\label{Ef.ineq}
		{\bf M}(\partial E_f) \leq \int_G {\bf M}(\partial ((g^{-1})_\# \La^*))~d\mu(g) = {\bf M}(\partial \La^*).%={\bf M}(S^*\vert_{B})+{\bf M}(D),
	\end{equation}
	Combining this with ${\bf M}(E_f)\leq {\bf M}([[M]])$, it's clear that $E_f$ is a normal current.
	By \cite[4.5.9(12)]{federer} we have that $\partial (f^{-1}[\lambda,1])$ is rectifiable for almost all $\lambda\in[0,1]$, which implies $\La_\lambda\in\C^G(M)$ for almost all $\lambda\in[0,1]$.
	
	Then by \cite[4.5.9(13)]{federer} and (\ref{Ef.ineq}), we have $\partial E_f =\int_0^1 \partial [[\La_\lambda]] d\lambda,$ and
	\begin{eqnarray}\label{Ef.ineq2}
		\int_0^1 {\bf M}(\partial [[\La_\lambda]])~d\lambda = {\bf M}(\partial E_f) \leq  {\bf M}(\partial \La^*). %={\bf M}(S^*\vert_{B})+{\bf M}(D).
	\end{eqnarray}
	Moreover, we have
	\begin{eqnarray*}
		\M(E_f) &=& \int_M f(x)~d\mH^{n+1}(x)
		\\
		&=& \int_0^1 \mH^{n+1}(f^{-1}[\lambda,1])~d\lambda
		\\
		&=& \int_0^1 \mH^{n+1}(\La_\lambda)~d\lambda;
		\\
		\M(E_f) &=& \int_M \int_G {\rm 1}_{\Clos(\La^*)}(g\cdot x) ~d\mu(g) ~d\mH^{n+1}(x)
		\\
		&=& \int_G \int_M {\rm 1}_{g^{-1}\cdot\Clos(\La^*)}(x)  ~d\mH^{n+1}(x) ~d\mu(g)
		\\
		&=& \int_G  \mH^{n+1}(g^{-1}\cdot\Clos(\La^*))   ~d\mu(g)
		\\
		&=& \mH^{n+1}(\La^*),
	\end{eqnarray*}
	where the last equality used the fact that $g\in G$ is isometry.
	Thus we have
	\begin{equation}\label{Ef.eq}
		\int_0^1 \mH^{n+1}(\La_\lambda)~d\lambda = \M(E_f) = \mH^{n+1}(\La^*).
	\end{equation}
	(\ref{Ef.ineq2}) together with (\ref{Ef.eq}) imply:
	\begin{equation}
		\int_0^1 \Ac (\La_\lambda)~d\lambda = \Ac(E_f)\leq \Ac(\La^*).
	\end{equation}
	
	 Since closed set $\mbox{spt}(\Om^*-\La^*)\subset B$, there exists $0<r<\rho$ such that $\Om^*=\La^*$ on $\sB^G_\rho(p)\setminus \sB^G_r(p)$.
	Thus for any $\lambda\in (0,1)$, we have $f={\rm 1}_{\mbox{Clos}(\La^*)}={\rm 1}_{\La_\lambda}$ on $\sB^G_\rho(p)\setminus \sB^G_r(p)$, and $\Om^*=\La^* = \La_\lambda= E_f$ on $\sB^G_\rho(p)\setminus \sB^G_r(p)$, which implies $\mbox{spt}(\Om^*-\La_\lambda)\subset B $.
	Since $\Om^*$ is $(G,\Ac)$-minimizing in $\sB^G_\rho(p)$, we have
	\begin{eqnarray*}
		\Ac(\Om^*)\leq \int_0^1 \Ac(\La_\lambda)~d\lambda \leq \Ac(\La^*) \leq \Ac(\Om^*).
	\end{eqnarray*}
	Thus we have $\Om^*$ is $\Ac$-minimizing in $\sB^G_\rho(p).$
\end{proof}

\begin{proposition}[Existence and properties of $(G,c)$-replacements]
\label{P:good-replacement-property}
Let $V\in\V_n^G(M)$ be $(G,c)$-almost minimizing in an open $G$-set $U \subset M$ and $K \subset U$ be a compact $G$-subset, then there exists $V^{*}\in \V_n^G(M)$, called \emph{a $(G,c)$-replacement of $V$ in $K$} such that
\begin{enumerate}
\item[(i)] $V\lc (M\backslash K) =V^{*}\lc (M\backslash K)$;
\item[(ii)] $-c \vol(K)\leq \|V\|(M)-\|V^{*}\|(M) \leq c \vol(K)$;
\item[(iii)] $V^{*}$ is $(G,c)$-almost minimizing in $U$;
\item[(iv)] $V^{*} =\lim_{i \to \infty} |\partial\Om^*_i|$ as varifolds for some $\Om^*_i\in\C^G(M)$ such that $\Om^*_i\in \mathfrak{a}^{c,G}_n(U; \ep_i, \de_i;\F)$ with $\ep_i, \de_i \to 0$; and $\Om^*_i$ locally minimizes $\Ac$ in $\interior(K)$;
\item[(v)] if $V$ has $c$-bounded first variation in $M$, then so does $V^*$.
\end{enumerate}
\end{proposition}

\begin{proof}
	Since $V\in \V_n^G(M)$ is $(G,c)$-almost minimizing in $U$, there exists a sequence $\Om_i\in \mathfrak{a}^{c,G}_n(U; \ep_i, \de_i;\F)$ with $\ep_i, \de_i \to 0$ such that $V=\lim_{i\to\infty}|\partial\Om_i|$.
	By Lemma \ref{L:minimisation}, there is a $G$-invariant $\Ac$-minimizer $\Om_i^* \in \C^G_{\Om_i}$ for each $i$.
	Since $\M(\partial\Om_i^*)$ is uniformly bounded, by compactness theorem, $|\partial\Om_i^*|$ converge to a $G$-varifold $V^*\in \mathcal{V}_n^G(M)$ after passing to a subsequence, i.e. $V^*=\lim |\partial\Om_i^*|$.

We now show that $V^*=\lim |\partial\Om_i^*|$ is exactly what we want. Since $\mbox{spt}(\Om_i^*-\Om_i)\subset K$, we have (i) $\mbox{spt}(V^{*}-V)\subset K$.

Due to the fact that $\Om_i\in\mathfrak{a}^{c,G}_n(U;\epsilon_i,\delta_i;\F) $ and $\Om_i^*$ is a $\Ac$ minimizer in $\mathcal{C}^G_{\Om_i}$, the following inequality holds:
	\begin{eqnarray*}
		{\bf M}(\partial\Om_i)-\epsilon_i-c \vol(K) &\leq & {\bf M}(\partial\Om_i)-\epsilon_i-c \mH^{n+1}(\Om_i)+c\mH^{n+1}(\Om_i^*)
		\\
		&=& \Ac(\Om_i)-\epsilon_i +c\mH^{n+1}(\Om_i^*)
		\\
		&\leq & \Ac(\Om_i^*)+c\mH^{n+1}(\Om_i^*)
		\\
		&=& {\bf M}(\partial\Om_i^*)
		\\
		&\leq & \Ac(\Om_i) +c\mH^{n+1}(\Om_i^*)
		\\
		&=& {\bf M}(\partial\Om_i)-c \mH^{n+1}(\Om_i)+c\mH^{n+1}(\Om_i^*)
		\\
		&\leq & {\bf M}(\partial\Om_i)+c \vol(K).
	\end{eqnarray*}

	Hence (ii) $-c \vol(K)\leq \|V\|(M)-\|V^{*}\|(M) \leq c \vol(K)$ is true.

	As for (iii) and (iv), they clearly follow from the facts that $\Om^*_i\in \mathfrak{a}^{cG}_n(U;\epsilon_i,\delta_i;\F )$  and $\Om^*_i$ is locally $\Ac$-minimizing in $\interior(K)$ by Lemma \ref{L:minimisation} and \ref{L:minimisation-II}.
	Finally by (iii) and Lemma \ref{(G,c)-almost minimizing implies c-bounded first variation}, $V^*$ has $c$-bounded first variation in $U$. By (i) and a standard cutoff trick it is easy to show that $V^*$ has $c$-bounded first variation in $M$ whenever $V$ does.
\end{proof}

\begin{lemma}[Regularity of $(G,c)$-replacement]
\label{L:reg-replacement}
Let $2 \leq n \leq 6$. Under the same hypotheses as Proposition \ref{P:good-replacement-property}, if $\Sigma=\spt \|V^*\| \cap \interior(K)$, then
\begin{enumerate}
\item $\Sigma$ is a smooth, $G$-equivairnt almost embedded, stable $(G,c)$-boundary;
\item the density of $V^*$ is $1$ along $\mR(\Si)$ and $2$ along $\mS(\Si)$;
\item the restriction of the $(G,c)$-replacement $V^* \lc \interior(K)=\Sigma$.
\end{enumerate}
\end{lemma}
\begin{proof}
By Proposition \ref{P:good-replacement-property}(iv) and the regularity for local minimizers of the $\Ac$ functional (\cite[Theorem 2.14]{zhou and zhu}), we know that each $\partial\Om^*_i$ is a smooth, embedded, $G$-invariant $c$-boundary in $\interior(K)$.
Moreover, if $\partial\Om_i^*$ is not $G$-stable in $\interior(K)$, then there exists $X\in \mathfrak{X}^G(\interior(K))$ such that
$$ \Ac((\Phi_{X}(t))_{\#}(\Om_i^*))<\Ac(\Om_i^*),\quad t\in (0,\tau ),$$
where $\Phi_{X}(t)$ are $G$-equivariant diffeomorphisms generated by $X$.
For $t$ small enough, we have $\F(\partial ((\Phi_{X}(t))_{\#}(\Om_i^*)),\partial\Om_i^*)<\delta_i$ implying $(\Phi_{X}(t))_{\#}(\Om_i^*))<\Ac(\Om_i^*)\in \C^G_{\Om_i}$, which is a contradiction to the $\Ac$-minimizing property of $\Om_i^*$ in $\C^G_{\Om_i}$.
Thus, $\partial\Om_i^* $ is $G$-stable in $\interior(K)$ and additionally stable in $\interior(K)$ by Lemma \ref{G-stable}.
The lemma then follows from the Compactness Theorem \ref{T:compactness}.
\end{proof}

\medskip
\subsection{Tangent cones and blowups}~

By the main theorem of \cite{moore}, there is an orthogonal representation of $G$ on some Euclidean space $\R^L$ and an isometric embedding from $M$ into $\R^L$ which is $G$-equivariant.
Thus we can regard $g\in G$ as an orthogonal $L$-matrix.
Given $p\in M, r>0$, let $\bleta_{p,r}: \R^L\to\R^L$ be the dilation defined by $\bleta_{p, r}(x)=\frac{x-p}{r}$.
The next lemma shows the splitting property of the blowups.

\begin{lemma}\label{L:splitting}
	Let $2 \leq n\leq 6$ and $\mathcal{H}^{n-1}(M\backslash M^{reg})=0$. Suppose $V \in \V_n^G(M)$ has $c$-bounded first variation in $M$ and is {\it $(G,c)$-almost minimizing in regular annuli}. Let $\overline{V} = \lim_{i\to\infty} (\bleta_{p_i,r_i})_\# V $, where  $p_i\to p \in \spt\|V\|\cap M^{reg}$ and $r_i>0$ with $r_i\to 0$.
	Then $\overline{V}=T_p(G\cdot p)\times W$ for some rectifiable varifold $W\in \V_{n-\dim_p}({\bf N}_p(G\cdot p))$.
\end{lemma}
\begin{proof}
	Firstly, since $V$ has $c$-bounded first variation in $M$ and is {\it $(G,c)$-almost minimizing in regular annuli}, one can use Proposition \ref{P:good-replacement-property} to get a volume ratio bound as \cite[Lemma 6.4, 6.5]{WTR}, which implies that $V$ as well as $\overline{V}$ is rectifiable.

	For any $w\in T_p(G\cdot p)$, there exists a curve $g(t)$ in $G$ such that $g(0)=e$ and $w=\frac{d}{dt}\big\vert_{t=0} g(t)\cdot p$.
	By the orthogonal representation of $G$ on $\R^L$, we can write $g(t)=I_L+tA(t)\subset O(\R^L)$, where $I_L$ is the identity map in $\R^L$ and $A(t)$ is a continuous curve in $L$-matrix such that $A(0)\cdot p =w$.
	A straightforward calculation shows that
	\begin{eqnarray*}
		(\bleta_{p_i, r_i}\circ g(r_i))(x) &=& \frac{g(r_i)\cdot x - p_i}{r_i}
		\\
		&=& \frac{g(r_i)\cdot x - g(r_i)\cdot p_i}{r_i}+\frac{g(r_i)\cdot p_i -  p_i}{r_i}
		\\
		&=& (\btau_{A(r_i)\cdot p_i}\circ g(r_i)\circ \bleta_{p_i, r_i})(x),
	\end{eqnarray*}
	where $\btau_{q}(y)=y+q$.
	Since $V$ is $G$-invariant, $A(r_i)\cdot p_i\to w$, and $g(r_i)\to e$, we have
\begin{eqnarray*}
			\overline{V} &=& \lim_{i\to\infty} (\bleta_{p_i,r_i})_\# V
			\\
			&=& \lim_{i\to \infty} (\bleta_{p_i, r_i})_\# (g(r_i))_\# V
			\\
			&=& \lim_{i\to \infty} (\btau_{A(r_i)\cdot p_i}\circ g(r_i)\circ \bleta_{p_i, r_i})_\# V
            \\
            &=& (\btau_{w})_\# \overline{V}.
		\end{eqnarray*}
	Thus $\overline{V}$ is invariant under the translation along $T_p(G\cdot p)$.
	
	Since $T_pM=T_p(G\cdot p)\times {\bf N}_p(G\cdot p)$ and $\overline{V}$ is rectifiable, we have $\overline{V}=T_p(G\cdot p)\times W$ for $W=\overline{V}\cap {\bf N}_p(G\cdot p)$.
\end{proof}

\begin{remark}\label{R:splitting}
	Suppose $V_i$ is a $(G,c)$-replacement of $V$ in some small $G$-annulus $\an_i$ with $\lim_{i\to \infty}\bleta_{p_i,r_i}(\an_i)= T_p(G\cdot p)\times {\rm An}\subset T_pM$, where ${\rm An}$ is a $G_p$-annulus in ${\bf N}_p(G\cdot p)$.
	One can also have the splitting property for $\lim_{i\to\infty} (\bleta_{p_i,r_i})_\# V_i $ by the argument above.
\end{remark}

\medskip
Combining the splitting property and the $(G,c)$-replacement, we can show the following proposition which classifies the tangent cone of the $c$-min-max varifold.

\begin{proposition}[Tangent cones are planes]
\label{P:tangent-cone}
	Let $2 \leq n \leq 6$ and $\mathcal{H}^{n-1}(M\backslash M^{reg})=0$. Suppose $V \in \V_n^G(M)$ has $c$-bounded first variation in $M$ and is {\it $(G,c)$-almost minimizing in regular annuli}. Then $V$ is integer rectifiable. Moreover, for any $C \in \VarTan(V,p)$ with $p \in \spt\|V\|\cap M^{reg}$,
	\begin{equation}
	\label{E:tangent cones are planes}
	C= \Theta^n (\|V\|, p) |S| \text{ for some $n$-plane $S \subset T_p M$ where $\Theta^n(\|V\|, p)\in\N$}.
	\end{equation}
\end{proposition}

\begin{proof}
	Let $r_i\to 0$ be a sequence such that $C$ is the varifold limit:
	\[ C=\lim_{i\to\infty} (\bleta_{p, r_i})_{\#} V. \]
	First we know $C$ is stationary in $T_p M$.

	By Lemma \ref{L:splitting}, $C$ has the form $C=T_p(G\cdot p)\times W$ for $W\in \V_{n-\dim_p}({\bf N}_p(G\cdot p))$.
	Moreover, if $g\in G_p$ (i.e. $g\cdot p=p$), then
	\begin{eqnarray*}
		(g\circ \bleta_{p, r_i})(x) = \frac{g\cdot x - g\cdot p}{r_i} = (\bleta_{p, r_i}\circ g)(x),
	\end{eqnarray*}
	which implies $C=g_\# C,~\forall g\in G_p$.
	Thus $W\in \V_{n-\dim_p}^{G_p}({\bf N}_p(G\cdot p)) $.

	To prove the proposition, we only need to show that $W$ is an integer multiple of some $(n-\dim_p)$-plane in ${\bf N}_p(G\cdot p)$.
		
		The rest of the proof has essentially the same idea as \cite[Lemma 5.10]{zhou and zhu}.
		First, $W$ is stationary in ${\bf N}_p(G\cdot p)$ by the product structure of $T_pM$ since $C$ is stationary in $T_pM$.
		
		Next, we will show that $W$ has the good $G_p$-replacement property (\cite[Proposition 6.1]{Liu}) in any open $G_p$-set $D\subset {\bf N}_p(G\cdot p)$.
		Fix a bounded open $G_p$-set $D\subset {\bf N}_p(G\cdot p)$ and an arbitrary $x\in D$.
		Take any $G_p$-annulus ${\rm An}={\rm An}^{G_p}(x,s,t)\subset D$ with $t\leq 1$, and denote ${\rm An}_i=(\bleta_{p, r_i})^{-1}({\rm An})$.
		As in \cite[Lemma 5.10]{zhou and zhu}, we identify $\bleta_{p, r_i}(M) $ with $T_pM$ on compact set for $i$ large, by the locally uniformly convergence $\bleta_{p, r_i}(M)\to T_pM$.
		Moreover, since $\bleta_{p, r_i}$ is $G_{p}$-equivariant, one can regard ${\rm An}_i$ as a $G_{p}$-annulus in $B_{p}$ for $i$ large, where $B_{p}$ is a slice of $G\cdot p$ at $p$.
		Denote $\an_i=G\cdot {\rm An}_i$ to be a $G$-annulus in $M$, then $\an_i\cap B_{p}={\rm An}_i$.
		For every $i$ large, one can have $\Clos(\an_i)\subset U $ and apply Proposition \ref{P:good-replacement-property} to get a $(G,c)$-replacement $V_i^*$ of $V$ in $\Clos(\an_i)$, where $U \subset M$ is an open $G$-set such that $V$ is $(G,c)$-almost minimizing in $U$.
		Denote $\overline{V}_i^*=(\bleta_{p, r_i})_\#V_i^*$ and $\overline{V}_i=(\bleta_{p, r_i})_\#V$, then after passing to a subsequence we have
		$$ C'=\lim_{i\to\infty} \overline{V}_i^*\in \V_n^{G_p}(T_pM).$$
		On the other hand, we have $\lim_{i\to \infty}\bleta_{p,r_i}(\an_i)= T_p(G\cdot p)\times {\rm An}\subset T_pM$, by Remark \ref{R:splitting}, $C'$ also has the splitting property:
		$$ C' = T_p(G\cdot p)\times W', $$
		where $W'\in \V_{n-\dim_p}^{G_p}({\bf N}_p(G\cdot p))$.
		Combining the Proposition \ref{P:good-replacement-property}, Lemma \ref{L:reg-replacement} and the Compactness Theorem \ref{T:compactness}, just as the argument in the proof of \cite[Lemma 5.10]{zhou and zhu}, one can see that $W'$ is a good $G_p$-replacement of $W$ in ${\rm An}$.
%$C' \lc {T_p(G\cdot p)\times {\rm An}}$ is an embedded stable $G$-invariant minimal hypersurface and $W'$ is a good $G_p$-replacement of $W$ in ${\rm An}$.
		By Proposition \ref{P:good-replacement-property}(iii), we can repeat the procedure above for a finite times and produce a good $G_p$-replacement $W^{(k)}$ of $W^{(k-1)}$ in any other $G_p$-annulus ${\rm An}^{(k)}\subset D$ with outer radius no more than $1$.
		
		Finally, $W$, as well as $C$, is a multiple of some complete minimal hypersurface by the regularity result \cite[Proposition 6.2]{Liu} and the Euclidean volume growth of $C$ coming from the monotonicity formula.
		(One should notice that the regularity results in \cite[Section 5, 6]{Liu} do not need $G$ to be connected by using Lemma \ref{G-stable} in the proof of \cite[Lemma 5.3]{Liu}.)
		Furthermore, $C$ is a cone by \cite[Theorem 19.3]{simon}, and hence, $\spt\|C\|$ must be a $n$-plane. 	
\end{proof}

\medskip
Using the splitting property, one can have the following lemma which is a modified version of \cite[Lemma 5.10]{zhou and zhu}.

\begin{lemma}\label{L:blowup is regular-1}
	Let $2 \leq n \leq 6$ and $\mathcal{H}^{n-1}(M\backslash M^{reg})=0$, $U\subset M$ be an open $G$-set, and $V\in\V_n^G(M)$ be a $(G,c)$-almost minimizing varifold in $U$.
	Given a sequence $p_i\in U$ with $p_i\to p\in U\cap M^{reg}$, and a sequence $r_i>0$ with $r_i\to 0$.
	Let $\overline{V}=\lim (\bleta_{p_i, r_i})_\# V$ be the varifold limit.
	Then $\overline{V}$ is an integer multiple of some complete embedded minimal hypersurface $\Si$ in $T_p M$, and moreover, $\Si$ is proper.
\end{lemma}

\begin{proof}
	By the openness of $M^{reg}$ and $p_i\to p$, we can suppose $p_i\in M^{reg}$. Hence $\{p_i\}_{i=1}^\infty$ and $p$ have the same orbit type, there exist a sequence $\{g_i\}_{i=1}^\infty\subset G$ such that $G_{p_i}=g_i\cdot G_p\cdot g_i^{-1}$, where $G_{p_i}$ is the isotropy group of $p_i$.
	By the compactness of $G$, we can suppose $g_i\to g_0\in G$ after passing to a subsequence.
	Moreover, it's clear that $g_0\cdot G_p\cdot g_0^{-1}= G_p$ since $p_i\to p$ and conjugate transformations are isomorphisms.
	As a result, one can see that $\bleta_{p_i, r_i} $ is $G_{p_i}$-equivariant:
		\begin{eqnarray*}
			h\circ \bleta_{p_i, r_i}(x) = \frac{h\cdot x-h\cdot p_i}{r_i} = \frac{h\cdot x-p_i}{r_i} = \bleta_{p_i, r_i}\circ h( x), \quad \forall h\in G_{p_i}.
		\end{eqnarray*}
		Now for any $g\in G_p$, let $h_i=g_i\cdot g_0^{-1}\cdot g\cdot g_0\cdot g_i^{-1}\in G_{p_i}$.
		It's clear that $h_i\to g$.
		Hence we have
		\begin{eqnarray*}
			g_\#\overline{V} &=& \lim_{i\to \infty} (h_i)_\# (\bleta_{p_i, r_i})_\# V
			\\
			&=& \lim_{i\to \infty} (\bleta_{p_i, r_i})_\# (h_i)_\# V
			\\
			&=& \lim_{i\to \infty} (\bleta_{p_i, r_i})_\# V = \overline{V},
		\end{eqnarray*}%\widetilde{p}_i
		which implies that $\overline{V}\in \V^{G_p}_n(T_pM)$.
		
		Since $V$ is $(G,c)$-almost minimizing in $U$, by Lemma \ref{(G,c)-almost minimizing implies c-bounded first variation}, $V$ has $c$-bounded first variation in $U$.
		Thus we can apply Lemma \ref{L:splitting} to get the the splitting property of $\overline{V}$:
		$$ \overline{V} = T_p(G\cdot p)\times W,$$
		where $W\in \V_{n-\dim_p}({\bf N}_p(G\cdot p))$.
		Since $T_p(G\cdot p)$ is $G_p$-invariant as varifold, one can see that $W\in \V_{n-\dim_p}^{G_p}({\bf N}_p(G\cdot p))$.
		To prove the lemma, we only need to show that $W$ is an integer multiple of some complete embedded minimal hypersurface in ${\bf N}_p(G\cdot p)$.
		
		The rest of the proof has essentially the same idea as Proposition \ref{P:tangent-cone}.
		One just need to notice that, $V$ has $c$-bounded first variation in $U$ implies that the blowup $\overline{V}=\lim (\bleta_{p_i, r_i})_\# V$ is stationary in $T_pM$.
		Thus $W$ is stationary in ${\bf N}_p(G\cdot p)$ by the product structure of $T_pM$.
		Moreover, since $\bleta_{p_i, r_i}$ is $G_{p_i}$-equivariant and $G_{p_i}\to G_p$, one can regard ${\rm An}_i=(\bleta_{p_i, r_i})^{-1}({\rm An})$ as a $G_{p_i}$-annulus in $B_{p_i}$ for $i$ large, where $B_{p_i}$ is a slice of $G\cdot p$ at $p_i$.
\end{proof}

\medskip
\subsection{Main regularity}~

Now we have the following regularity result for $(G,c)$-min-max varifold.

\begin{theorem}[Main regularity]\label{T:main-regularity}
	Let $2\leq n\leq 6$, and $(M^{n+1}, g_{_M})$ be an $(n+1)$-dimensional smooth, closed Riemannian manifold with a compact Lie group $G$ acting as isometries of cohomogeneity ${\rm Cohom}(G)\geq 3$.
	Suppose  the union of non-principal orbits $M\setminus M^{reg}$ is a smooth embedded submanifold of $M$ without boundary and ${\rm dim}(M\setminus M^{reg})\leq n-2 $.
    If $V \in \V_n^G(M)$ is a varifold which
	\begin{enumerate}
	\item has $c$-bounded first variation in $M$ and
	\item is $(G,c)$-almost minimizing in regular annuli,
	\end{enumerate}
	then $V$ is induced by $\Sigma$, where
	\begin{itemize}
		\item[(i)] $\Sigma$ is a closed, $G$-equivariant almost embedded $G$-invariant $(G,c)$-hypersurface (possibly $G$-disconnected);
		\item[(ii)] the density of $V$ is exactly $1$ at the regular set $\mR(\Si)$ and $2$ at the touching set $\mS(\Si)$.
	\end{itemize}
\end{theorem}

\begin{remark}\label{the case of c}
We only focus on the case of $c>0$ in Theorem \ref{T:main-regularity}. Indeed the work of Liu \cite{Liu} under the smooth sweepouts settings and T. R. Wang \cite{WTR} under the Almgren-Pitts settings completely resolved the $c=0$ case of Theorem \ref{thm:main theorem}.
\end{remark}

\medskip
\begin{proof}
    Let $p\in \spt\|V\|$, by the assumption of ${\rm Cohom}(G)\geq 3$ and $M\setminus M^{reg}$ is a smooth embedded submanifold of $M$ without boundary and ${\rm dim}(M\setminus M^{reg})\leq n-2 $, then there exists $0<r_0<r(p)$ such that for any $0<r<r_0$, the mean curvature $H$ of $\partial B^G_r(P_p)\cap M$ in $M$ is greater than $c$. Here $r(p)$ is as in Definition \ref{Def:a.m. in regular annuli}.

    By the maximum principal \cite[Theorem 5]{white} see also \cite[Proposition 2.13]{zhou and zhu} and the convexity of $B^G_r(P_p)$, if $W\in \mathcal{V}_n(M)$ has $c$-bounded first variation in $B^G_r(P_p)\cap M$ and $W\llcorner B^G_r(P_p)\neq 0$, then
	\begin{equation}\label{Eq:convex ball}
		\emptyset \neq {\rm spt}(W)\cap \partial B^G_r(P_p)  = {\rm Clos}\big[{\rm spt}\|W\| \setminus {\rm Clos}(B^G_r(p)) \big] \cap \partial B^G_r(p).
	\end{equation}

	{\bf Step 1.}
	After adding $G$- in front of relevant objects, the first step of the proof of \cite[Theorem 6.1]{zhou and zhu} would carry over.
	
	{\bf Step 2.}
	We only need to modify the second step of the proof of \cite[Theorem 6.1]{zhou and zhu} in a few places.
	
	First, since every small $G$-annuli around $P_p$ is contained in $M^{reg}$ (see Section 5 (\ref{Eq:regular annuli})), ${\rm Cohom}(G)\geq 3$ and ${\rm dim}(M\setminus M^{reg})\leq n-2 $, Proposition \ref{P:tangent-cone}, Lemma \ref{L:blowup is regular-1} and all the results in the previous subsection can be applied to any $q\in \an(P_p,s,t)\in \ann_r(P_p)$.
	Hence we can use Lemma \ref{L:blowup is regular-1} in the place of \cite[Lemma 5.10]{zhou and zhu} and get the following result as Claim 3(A) in \cite[Page 475]{zhou and zhu} (noting $\Gamma\subset \Clos(\an(p,s,t))\cap M$):
	\begin{claim}
		Fix $x\in \mR(\Gamma)$, for any sequence of $x_i\to x$ with $x_i\in\mR(\Gamma)$ and $r_i\to 0$, we have
		$$ \lim_{i\to\infty} (\bleta_{x_i, r_i})_\# V^{**}=T_x\Si_1 \quad {\rm as~ varifolds}. $$
	\end{claim}
	Then, after considering the blowups $\bleta_{x_i, r_i}(\Sigma_2\cap B^G_{r_i/2}(z_i))$ in the proof of Claim 4(A) \cite[Page 476]{zhou and zhu}, we can show the gluing of the $(G,c)$-replacements is smooth on the overlap for sub-case (A).

	As for sub-case (B), we can also get the Claim 3(B) in \cite[Page 477]{zhou and zhu} by using Lemma \ref{L:blowup is regular-1} and considering the various blowups $(\bleta_{x_i, r_i})_\# V^{**}$, $\bleta_{x_i, r_i}(\Sigma_{1,1})$, $\bleta_{x_i, r_i}(\Sigma_{1,2})$, as well as the vector $v=\lim_{i\to\infty} \frac{x_i'-x_i}{r_i}$.

Hence, we can construct successive $G$-replacements $V^*$ and $V^{**}$ on two overlapping $G$-regular annuli $\an(P_p,s,t),~\an(P_p,s_1,s_2)$, ($0<s_1<s<s_2<t$), and glue them smoothly across $\partial B^G_{s_2}(P_p)$.

	{\bf Step 3.}
	As the third step in the proof of \cite[Theorem 6.1]{zhou and zhu}, we denote $V^{**}$ by $V^{**}_{s_1}$ and $\Sigma_2$ by $\Sigma_{s_1}$ to indicate the dependence on $s_1$.
    Since unique continuation hold for immersed CMC hypersurfaces, by Step 2 we have $\Sigma_{s_1}=\Sigma_1$ in $\an(P_p,s,s_2)$, and moreover, we have $\Sigma_{s'_1}=\Sigma_{s_1}$
    in $\an(P_p,s_1,s_2)$, for any $s'_1<s_1<s$. Hence
    $$\Sigma:=\bigcup_{0<s_1<s}\Sigma_{s_1}$$
    is a smooth, almost embedded, stable $(G,c)$-hypersurface in $(B^G_{s_2}(P_p)\backslash G\cdot P_p)\cap M$.
    By Proposition \ref{P:good-replacement-property}, $V^{**}_{s_1}$ has $c$-bounded first variation and uniformly bounded mass for all $0<s_1<s$. By a covering argument
     \cite[Lemma A.1]{Liu}, Weyl's Tube Formula  \cite[Corollary 4.11]{Gray} and the monotonicity formula \cite[Theorem 40.2]{simon}, we have
     $$\|V^{**}_{s_1}\|(B^G_r(P_p))\leq Cr^{n-{\rm dim}(G\cdot P_p)}\leq Cr^2$$
     for some uniform $C>0$. Therefore as $s_1\to 0$, the family $V^{**}_{s_1}$ will converge to a varifold $\tilde{V}\in \V_n^G(M)$, i.e
     $\tilde{V}=\lim_{s_1\to 0}V^{**}_{s_1}$, we have $\|\tilde{V}\|(G\cdot P_p)=0$ and
     \begin{eqnarray*}
       \tilde{V}=
       \begin{cases}
       \Sigma       & in \ \ (B^G_{s_2}(P_p)\backslash G\cdot P_p)\cap M \\
       V^*   & in \ \ M\backslash B^G_{s}(P_p) \\
       \end{cases}
     \end{eqnarray*}

	{\bf Step 4.}
    The regularity of $\tilde{V}$ at $G\cdot P_p$ comes from the \cite[Corollary 1.1]{Bellettini and Wickramasekera}.
    First, we have $\tilde{V}=\lim_{s_1\to 0}V^{**}_{s_1}$, $\|\tilde{V}\|(G\cdot P_p)=0$ and $\tilde{V}=\Sigma$ in $(B^G_{s_2}(P_p)\backslash G\cdot P_p)\cap M$. Then by modifying the \cite[Claim 1]{zhou and zhu} in step 1 of the proof of \cite[Theorem 6.1]{zhou and zhu}, we have that $\tilde{V}$ is a boundary of some $G$-invariant Caccioppoli set $\Omega$, i.e $\tilde{V}=|\partial \Omega|$. %, since $\tilde{V}=\lim_{s_1\to 0}V^{**}_{s_1}$, $\|\tilde{V}\|(G\cdot P_p)=0$ and $\tilde{V}=\Sigma$ in $(B^G_{s_2}(P_p)\backslash G\cdot P_p)\cap M$.

	One should notice that the center of $B^G_r(P_p)$ is $G\cdot P_p$, where $G\cdot P_p$ is a principle orbit if $p\in M^{reg} $ or a smooth embedded submanifold of $M$ without boundary if $p\in M\setminus M^{reg}$.
	But since ${\rm Cohom}(G)\geq 3$ and ${\rm dim}(M\setminus M^{reg})\leq n-2 $, we have $\dim(G\cdot P_p)\leq n-2$.
	Thus the extension argument is still valid here (c.f. the proof in \cite[Theorem 4.1]{HL75}), which gives the {\em first variation} for $\Omega$
    $$\de\Ac|_{\Om}(X)=0,\   X\in \X(M) \ with \ spt(X)\subset B^G_{s_2}(P_p),$$
    where $\tilde{V}=|\partial \Omega|$, $\Omega\in \C^G(M)$.

    Since $\dim(G\cdot P_p)\leq n-2$ and $\tilde{V}=\Sigma$ is a smooth, almost embedded, stable $(G,c)$-hypersurface in $(B^G_{s_2}(P_p)\backslash G\cdot P_p)\cap M$ , $\tilde{V}=|\partial \Omega|$ has no classical singularities in $B^G_r(P_p)$ (c.f. the \cite[Definition 1.5]{Bellettini and Wickramasekera}) and the condition (c) in \cite[Theorem 1.5]{Bellettini and Wickramasekera} is satisfied.
    Hence, the singularity of $\tilde{V}$ at $G\cdot P_p$ is removable by \cite[Corollary 1.1]{Bellettini and Wickramasekera}.
		
	{\bf Step 5.} Same with the final step in the proof of \cite[Theorem 6.1]{zhou and zhu} except using $G$-invariant objects.
\end{proof}

\vspace{0.5cm}
\bibliographystyle{amsbook}

\end{document}